\documentclass{elsarticle}
\usepackage{mathrsfs}
\usepackage{epsfig}
\usepackage{graphicx}
\usepackage{makecell,rotating}
\usepackage{color}
\usepackage[labelsep=quad,indention=10pt]{caption}
\usepackage[list=true]{subcaption}
\usepackage{hyperref}

\DeclareCaptionSubType*[arabic]{table}
\DeclareCaptionSubType*[arabic]{figure}
\usepackage{chngcntr}

\usepackage{amsmath}
\usepackage{amssymb}
\usepackage{amsfonts}
\usepackage{amsmath,amsfonts,amssymb}
\usepackage{fancyhdr}
\usepackage[top=1in, bottom=1in, left=1.25in, right=1.25in]{geometry}
\def \a {\alpha}
\def \D {\mathcal{D}}
\def \p {\partial}
\def \t {\theta}
\def \R {\mathbb{R}}
\def \e {\epsilon}
\def \g {\gamma}
\def \b {\beta}
\def \l {\lambda}
\def \L {\mathcal{L}}

\def \I {\uppercase\expandafter{\romannumeral1}}
\def \II {\uppercase\expandafter{\romannumeral2}}
\def \III {\uppercase\expandafter{\romannumeral3}}
\def \IV {\uppercase\expandafter{\romannumeral4}}
\def \V {\uppercase\expandafter{\romannumeral5}}
\def \VI {\uppercase\expandafter{\romannumeral6}}

\begin{document}
\newtheorem{definition}{Definition}
\renewcommand{\thedefinition}{\arabic{section}.\arabic{definition}}
\newtheorem{proposition}{Proposition}
\renewcommand{\theproposition}{\arabic{section}.\arabic{proposition}}
\newtheorem{theorem}{Theorem}
\renewcommand{\thetheorem}{\arabic{section}.\arabic{theorem}}
\newtheorem{lemma}{Lemma}
\renewcommand{\thelemma}{\arabic{section}.\arabic{lemma}}
\newtheorem{corollary}{Corollary}
\renewcommand{\thecorollary}{\arabic{section}.\arabic{corollary}}
\newtheorem{remark}{Remark}
\renewcommand{\theremark}{\arabic{section}.\arabic{remark}}
\renewcommand{\thefigure}{\arabic{section}.\arabic{figure}}
\renewcommand{\thetable}{\arabic{section}.\arabic{table}}

\journal{Applied Numerical Mathematics}

\begin{frontmatter}
\title{Spectral viscosity method with generalized Hermite functions for nonlinear conservation laws}

\author{ Xue Luo}
\address{School of Mathematics and Systems Science, Beihang University, Beijing, P. R. China, 100191 ({\tt xluo@buaa.edu.cn})}

\begin{abstract}
	In this paper, we propose new spectral viscosity methods based on the generalized Hermite functions for the solution of nonlinear scalar conservation laws in the whole line. It is shown rigorously that these schemes converge to the unique entropy solution by using compensated compactness arguments, under some conditions. The numerical experiments of the inviscid Burger's equation support our result, and it verifies the reasonableness of the conditions.
\end{abstract}


\end{frontmatter}

\section{Introduction}

\setcounter{equation}{0}

The spectral methods \cite{GO} approximate the exact solution of partial differential equations by seeking an ``good" projection in the linear subspace spanned by various orthogonal systems of special functions. The resulting spectral accuracy is highly preferred than any other numerical method, especially when the solution is known to be globally smooth enough. Therefore, they are very appropriate for the elliptic and parabolic equations, thanks to the regularization properties of the operators. When mentioning the nonlinear conservation laws, it is well known that the solution may develop spontaneous jump discontinuity, i.e., shock waves. This irregularity of the solution destroys not only the accuracy of the spectral approximations at the point of discontinuity, but also that in the entire computational domain. It causes the oscillations throughout the domain, which is the so-called Gibb's phenomenon. Moreover, the instability is induced in the nonlinear case. It is shown in \cite{Ta2} that the usual spectral approximate solution may not converge to the entropy solution, the physically relevant one.

Despite all these deficiencies, many mathematicians still pay their {efforts} to deal with these {issues}. The problems caused by the irregularity have already been solved for piecewise smooth functions in bounded domain or periodic piecewise smooth function in unbounded domain by filter techniques or reconstruction methods such as the Gegenbauer partial sum, see details in a series of papers \cite{GT}, \cite{GSSV}, \cite{GS}, \cite{V} and references therein. And the instability of the usual spectral approximations can be avoided by introducing the vanishing viscosity, which was first established by E. Tadmor \cite{Ta1}. The main idea of the spectral viscosity method is the use of artificial diffusion to stabilize the spectral computation without sacrificing its spectral accuracy. The periodic spectral {viscosity} method has been further investigated in \cite{MT}, \cite{Ta2} and \cite{Sch}, etc. The nonperiodic Legendre spectral viscosity method is first introduced by Y. Maday, et. al. \cite{MOT}. {H. Ma} proposed the nonperiodic Chebyshev-Legendre spectral viscosity method in \cite{Ma1}, \cite{Ma2}. For more literatures related to the spectral viscosity methods with various orthogonal basis in bounded domain, we refer the readers to \cite{CDT}, \cite{GeT}, \cite{GMT} and references therein.

As we know, a large amount of physical problems are modeled in unbounded domain. During the past two decades, {more attentions} were {attracted} to {the} numerical solutions of differential {equations  in} unbounded domains. Among the existing literature, the Hermite and Laguerre spectral methods are the most commonly used approaches based on orthogonal polynomials in infinite interval, referring to \cite{FK}, \cite{XW}. Although the Hermite polynomials appear to be a natural choice of orthogonal basis of $L^2(\mathbb{R})$, it is not as popular as Fourier series and Chebyshev polynomials, due to its poor resolution (see \cite{GO}) and the lack of the analogue of fast Fourier transformation (FFT), see \cite{B2}. However, it is shown in \cite{B} that the poor resolution can be remedied by a suitable choice of scaling factor. Some further investigations on the scaling factor can be found in \cite{Tang} {and Chapter} 7, \cite{STW}. Recently, a {practical} guideline of choosing the suitable scaling factors for Gaussian/super-Gaussian functions is summarized by {S. S.-T. Yau and the author} in \cite{LY}, where the Hermite spectral method is used to resolve the {posterior} conditional density function of the states in nonlinear filtering problems.

The literatures on the spectral method in unbounded domains have already been not as rich as those in bounded domains, let alone the spectral viscosity method in unbounded domains. As far as we know, J. Aguirre and J. Rivas \cite{AR} is the only paper that considered the spectral viscosity method based on the Hermite functions. However, they defined the Hermite functions in the weighted $L_w^2(\mathbb{R})$, where $w(x)={e^{x^2}}$. {No} scaling factor is introduced {there}. This essentially causes their involved theoretical proof of the convergence rate of their proposed scheme. {Also} it is more costly when they try to implement their scheme numerically.

In this paper, we shall revisit the nonlinear scalar conservation laws in $\R$:
\begin{equation}\label{eqn-nonlinear conservation}
	\left\{\begin{aligned}
		\frac{\partial u}{\partial t}+\frac{\partial f(u)}{\partial x}=&0,\quad x\in\mathbb{R},\,t>0\\
		u(x,0)=&u_0(x),\quad x\in\mathbb{R},
	\end{aligned}\right.
\end{equation}
where $f\in C^1$ is a smooth nonlinear function and $u_0\in L^\infty(\mathbb{R})$. In general, the spontaneous jump discontinuity may be developed. Therefore, we {can not} expect the classical solutions to this problem. Moreover, we restrict ourselves to the physically relevant weak solution, the entropy solution, by imposing the entropy condition
\begin{align}\label{eqn-entropy condition}
	\frac{\partial U(u)}{\partial t}+\frac{\partial F(u)}{\partial x}\leq0
\end{align}
in the sense of distributions, for all entropy pairs $(U, F)$, with $U\in C^2$ convex and $F'(u)=U'(u)f'(u)$, see \cite{Sm}.

We propose Hermite spectral viscosity methods based on generalized Hermite functions with two different viscosity terms. 
\begin{enumerate}
\item[(I)] with viscosity term $\e\p_x\D_xu$:
{The approximate solution $u_N$ is obtained by solving}
\begin{equation}\label{eqn-spectral scheme}
	\left\{\begin{aligned}
		\partial_tu_N+\partial_x(P_{N+1}f(u_N))-\epsilon_N\p_x\D_xQ_{m_N}u_N=&0,\quad x\in\mathbb{R},\,t\in(0,T),\\
		u_N(x,0)=&P_Nu_0(x),\quad x\in\mathbb{R},
	\end{aligned}\right.
\end{equation}
where {$P_{N+1}$ is an $L^2$-orthogonal projection operator defined in \eqref{eqn-L2projection}, $\epsilon_N\rightarrow0$} is a positive parameter as $N$ tends to $\infty$, and $Q_{m_N}$ is a viscosity operator which modifies only the high modes of the Fourier-Hermite expansion{. That} is,
\begin{align}\label{eqn-Q}
	Q_{m_N}\left(\sum_{k=0}^N\hat\phi_k(t)H_k^\alpha(x)\right)=\sum_{k=0}^N\hat q_k\hat\phi_k(t)H_k^\alpha(x),
\end{align}
with
\begin{equation}\label{eqn-q}
	\left\{\begin{aligned}
		\hat q_k=0,\quad&\textup{if}\quad k\leq m_N\\
		1-\frac{m_N}k\leq\hat q_k<1,\quad&\textup{if}\quad k>m_N
	\end{aligned}\right.,
\end{equation}
and $m_N<N$ is a positive integer {tending} to $\infty$ {as $N$ tends to $\infty$}{, where $H_k^\a(x)$ are the generalized Hermite functions defined in \eqref{new hermite}.}
\item[(II)] with viscosity term $\e\L_\a u$: {The approximate solution $v_N$ is obtained by solving}
\begin{equation}\label{eqn-spectral scheme-v}
	\left\{\begin{aligned}
		\partial_tv_N+\partial_x(P_{N+1}f(v_N))+\epsilon_N\L_\a v_N=&0,\quad x\in\mathbb{R},\,t\in(0,T),\\
		v_N(x,0)=&P_Nu_0(x),\quad x\in\mathbb{R},
	\end{aligned}\right.
\end{equation}
where {$P_{N+1}$, $\epsilon_N$ are the same as those in (I), and} $\L_\a$ is defined in \eqref{S-L}.
\end{enumerate}

{Nevertheless, compared} to the scheme in \cite{AR}, the schemes in our paper have at least two advantages:
\begin{itemize}
	\item {Our scheme (II) is numerically stable}, due to its symmetry and positivity, while the stability of the scheme in \cite{AR} and {our scheme (I) can not be} guaranteed;
	\item Our schemes can be implemented {efficiently} with the help of the scaling factor. The better resolution and fewer oscillations are retained with much smaller truncation terms $N$ even without viscosity.
\end{itemize}

{In this paper, we} {shall develop two efficient} schemes to solve the nonlinear conservation laws in $\R$. The convergences { of the schemes have been shown} under some reasonable condtions \eqref{eqn-condition on xu_N} or \eqref{cond-x2v_N}. It is hard to tell whether these conditions are weaker or {stronger} than the one in \cite{AR}, i.e., $||xu_N||_\infty<C$, independent of $N$, due to the unbounded domain $\R\times(0,T)$.

The paper is organized as follows. In section 2, we give the definition of the generalized Hermite functions and their properties. The new Hermite spectral viscosity methods are proposed { and their convergences have been rigorously shown} in section 3. {In section 4, the inviscid Burger's equation has been numerically solved by our schemes.} The reasonableness of the conditions { in the convergence theorems} have been verified numerically.

\section{Generalized Hermite functions}

\setcounter{equation}{0}

In this section, we introduce the generalized Hermite functions and
derive {their properties inherited} from the Hermite
polynomials. 

Let $L^2(\mathbb{R})$ be the Lebesgue space, equipped with the
norm $||\cdot||=(\int_{\mathbb{R}}|\cdot|^2 dx)^{\frac12}$ and the
scalar product $\langle\cdot,\cdot\rangle$. {Moreover, we shall denote $||\circ||_\infty=||\circ||_{L^\infty(\mathbb{R}\times(0,T))}=\sup_{\mathbb{R}\times[0,T]}|\circ|$ and $||\circ||^2_{L^2(0,T;L^2(\mathbb{R}))}:=\int_0^T||\circ||^2dt$.}

In the sequel, we shall follow the {conventional notations} in the asymptotic analysis{:}
$a\sim b$ means that there {exist} some generic constants $C_1,C_2>0$ such
that $C_1a\leq b\leq C_2a$; $a\lesssim b$ means that there exists
some generic constant $C_3>0$ such that $a\leq C_3b$.

Let $\mathcal{H}_n(x)$  be the physical Hermite polynomials, i.e.,
$\mathcal{H}_n(x)=(-1)^ne^{x^2}\partial_x^ne^{-x^2}$, $n\geq0$. The three-term recurrence
\begin{align}\label{recurrence}
    \mathcal{H}_0\equiv1,\quad \mathcal{H}_1(x)=2x\quad\textup{and}\quad
    \mathcal{H}_{n+1}(x)=2x\mathcal{H}_n(x)-2n\mathcal{H}_{n-1}(x).
\end{align}
is {handy} in implementation. One of the well-known and useful
fact of Hermite polynomials is that they are mutually orthogonal
with respect to the weight $w(x)=e^{-x^2}$. We define the generalized Hermite functions with the scaling factor $\alpha>0$ as
\begin{align}\label{new hermite}
    H_n^{\alpha}(x)=\left(\frac{\alpha}{2^nn!\sqrt{\pi}}\right)^{\frac12}\mathcal{H}_n(\alpha x)e^{-\frac12\alpha^2x^2},
\end{align}
for $n\geq0$. It is readily to derive the following properties for
the generalized Hermite functions (\ref{new hermite}):
\begin{enumerate}
	\item [$\blacksquare$] \label{orthogonal of H_n^alpha} The $\{H_n^{\alpha}(x)\}_{n\in\mathbb{Z^+}}$ forms an orthonormal basis of
    $L^2(\mathbb{R})$, i.e.
        \begin{align}\label{orthogonal}
            \int_{\mathbb{R}}H_n^{\alpha}(x)H_m^{\alpha}(x)dx=\delta_{nm},
        \end{align}
    where $\delta_{nm}$ is the Kronecker function.
    \item [$\blacksquare$]\label{eigenvalue} $H_n^{\alpha}(x)$ is the $n${th} eigenfunction of the following Strum-Liouville problem
        {\begin{align}\label{S-L}
            \mathcal{L}_\alpha u(x):=-e^{\frac12\alpha^2x^2}\frac d{dx}\left(e^{-\alpha^2x^2}\frac d{dx}\left(e^{\frac12\alpha^2x^2}u(x)\right)\right)=\lambda_nu(x),
        \end{align}}
    with the corresponding eigenvalue
	\begin{equation}\label{eqn-lambda}
		\lambda_n=2\alpha^2n.
	\end{equation}
    \item [$\blacksquare$] By convention, $H_n^{\alpha}\equiv0$, for $n<0$. For $n\geq0$, the three-term recurrence is inherited from the
    Hermite polynomials:
    \begin{align}\label{recurrence for RT hermite function}
        xH_n^{\alpha}(x)=&\frac{\sqrt{\lambda_{n+1}}}{2\alpha^2}H_{n+1}^{\alpha}(x)
       +\frac{\sqrt{\lambda_n}}{2\alpha^2}H_{n-1}^{\alpha}(x).
    \end{align}
    \item [$\blacksquare$] The derivative of $H_n^{\alpha}(x)$ with respect to $x$ {gives}
    \begin{align}\label{derivative_x}
        {\frac d{dx}}H_n^{\alpha}(x)
       =&-\frac{\sqrt{\lambda_{n+1}}}2H_{n+1}^{\alpha}(x)
            +\frac{\sqrt{\lambda_n}}2H_{n-1}^{\alpha}(x).
    \end{align}
	For convenience, let $\mathcal{D}_x={\frac d{dx}}+\alpha^2x$. Then
	\begin{align}\label{eqn-D_xH_n}		\mathcal{D}_xH_n^\alpha(x)=\sqrt{2\alpha^2n}H_{n-1}^\alpha(x)=\sqrt{\lambda_n}H_{n-1}^\alpha(x).
	\end{align}
    \item [$\blacksquare$] The ``orthogonality" of $\{\mathcal{D}_xH_n^\alpha(x)\}_{n\in\mathbb{Z}^+}$ follows immediately from \eqref{orthogonal}, i.e.,
\begin{equation}\label{eqn-orthogonality of D_xH}
\int_{\mathbb{R}}\mathcal{D}_xH_n^\alpha(x)\mathcal{D}_xH_m^\alpha(x)dx
	= 2\alpha^2n\delta_{nm}=\lambda_n\delta_{nm}{.}
\end{equation}
\end{enumerate}

Any function $u(x)\in L^2(\mathbb{R})$ can be written in the
form
\begin{equation}\label{Hermite representation}
    u(x)=\sum_{n=0}^\infty\hat{u}_nH_n^{\alpha}(x),
\end{equation}
with
\begin{equation}\label{eqn-FH coefficient}
    \hat{u}_n=\int_{\mathbb{R}}u(x)H_n^{\alpha}(x)dx{,}
\end{equation}
where $\{\hat{u}_n\}_{n=0}^\infty$ are the Fourier-Hermite
coefficients.

Let us denote the linear subspace of $L^2(\mathbb{R})$ spanned by the first $N$ {generalized} Hermite functions by
\begin{equation}\label{RN}
	\mathcal{R}_N:=\textup{span}\{H_0^\alpha(x),\cdots,H_N^\alpha(x)\}.
\end{equation}

\begin{remark}\label{remark-equivalence of derivatives}
	Actually, we have the norms $||{\frac d{dx}}\phi||$ controlled by $||\mathcal{D}_x\phi||$ and $||\phi||$, for any $\phi\in\mathcal{R}_N$. Let us consider
{\begin{align}\label{eqn-equivalence of d_x and D_x}\notag
	\left|\left|\frac d{dx}\phi\right|\right|^2=&\left|\left|\sum_{k=0}^N\hat{\phi}_k\left(-\frac{\sqrt{\lambda_{k+1}}}2H_{k+1}^\alpha+\frac{\sqrt{\lambda_k}}2H_{k-1}^\alpha\right)\right|\right|^2\\\notag
	=&\sum_{k,l=0}^N\hat{\phi}_k\hat{\phi}_l\int_{\mathbb{R}}\left(\frac{\sqrt{\lambda_{k+1}}}2H_{k+1}^\alpha+\frac{\sqrt{\lambda_k}}2H_{k-1}^\alpha\right)\left(-\frac{\sqrt{\lambda_{l+1}}}2H_{l+1}^\alpha+\frac{\sqrt{\lambda_l}}2H_{l-1}^\alpha\right)dx\\\notag
\overset{\eqref{orthogonal}}=&\frac14\sum_{k=0}^N\hat{\phi}_k^2(\lambda_{k+1}+\lambda_k)-\frac14\sum_{k=0}^{N-2}\hat{\phi}_{k+2}\hat{\phi}_k\sqrt{\lambda_{k+2}\lambda_{k+1}}-\frac14\sum_{k=2}^N\hat{\phi}_k\hat{\phi}_{k-2}\sqrt{\lambda_k\lambda_{k-1}}\\
\leq&\frac12\sum_{k=0}^N\hat{\phi}_k^2\lambda_{k+1}+\frac12\sum_{k=0}^N\hat{\phi}_k^2\lambda_k
\overset{\eqref{eqn-lambda}}=\sum_{k=0}^N\hat{\phi}_k^2\lambda_k+\alpha^2\sum_{k=0}^N\hat{\phi}_k^2
=||\mathcal{D}_x\phi||^2+\a^2||\phi||^2,
\end{align}
where the inequality follows from the fact that
\begin{align*}
	\left|\hat\phi_{k+2}\hat\phi_k\sqrt{\lambda_{k+2}\lambda_{k+1}}\right|
	\leq&\frac12\left(\hat\phi_{k+2}^2\lambda_{k+2}+\hat\phi_k^2\lambda_{k+1}\right),
\end{align*}
for $k=0,\cdots,N-2$.} Similarly, we can get
\begin{equation}\label{eqn-equivalence of xphi}
	||x\phi||^2\lesssim\frac1{\a^4}\left[||\D_x\phi||^2+\a^2||\phi||^2\right].
\end{equation}
\end{remark}

We define the $L^2$-orthogonal projection $P_N^\alpha:\,L^2(\mathbb{R})\rightarrow\mathcal{R}_N${:} given $v\in L^2(\mathbb{R})$, {we have}
\begin{equation}\label{eqn-L2projection}
	\langle v-P_N^\alpha v,\phi\rangle=0,
\end{equation}
for all $\phi\in\mathcal{R}_N$. More precisely, {it can be written as}
\[
	P_N^\a v(x)=\sum_{n=0}^N\hat{v}_nH_n^\alpha(x),
\]
where $\hat{v}_n$, $n=0,\ldots,N$, are the Fourier-Hermite coefficients defined in {\eqref{eqn-FH coefficient}}.

{To establish the convergence rate of the Hermite spectral
method, we shall also state the convergence rate of the orthogonal approximation.} The error estimate of the orthogonal projection onto $\mathcal{R}_N$ is readily shown in Theorem 4.2, \cite{SW} for $\alpha=1$ and it can be trivially extended for $\alpha>0${. }
\begin{lemma}\label{lemma-orthogonal error}
	For any $\mathcal{D}_x^mu\in L^2(\mathbb{R})$ with $m\geq0$,
\[
	||\mathcal{D}_x^l(u-P_N^\alpha u)||\lesssim \alpha^{l-m}N^{\frac{l-m}2}||\mathcal{D}_x^mu||,\quad0\leq l\leq m.
\]
\end{lemma}
In the sequel, the superscript $\alpha$ in $P_N^\alpha$ will be dropped if no confusion will arise.

\section{The Hermite spectral viscosity method}
\setcounter{equation}{0}
\setcounter{theorem}{0}

It is well known that the entropy solution to \eqref{eqn-nonlinear conservation} can be obtained as the limit when the {artificially} introduced viscosity term vanishes. In this section, we shall introduce two appropriate viscosity terms $\epsilon\p_x\D_xu$ and  $\e\L_\a u$. {The convergences of both schemes {will} be shown under the assumption that the approximate solutions are uniformly bounded in $L^\infty$ norm. It is {well known} that the numerical solution {is bounded} in a finite interval, but not that in unbounded domain. {This interesting question will not be discussed in this paper}.} 

With the viscosity term $\e\p_x\D_xu$, the viscosity operator $Q_{m_N}$ has been introduced in the spectral scheme as in \cite{AR}, where only the high frequency terms appear in the artificial viscosity. The convergence of this scheme has been shown under the condition \eqref{eqn-condition on xu_N}. The reasonableness of this condition { in the inviscid Burger's equation} has been verified in Table \ref{table-1}.

\subsection{With viscosity term $\e\p_x\D_xu$}
 Let us {discuss} the viscosity operator $Q_{m_N}$ {first}.
\begin{lemma}\label{lemma-D_x<D_xQ}
	Let $Q_{m_N}$ be defined as in \eqref{eqn-Q} {and} \eqref{eqn-q}. Then
\begin{align}\label{eqn-D_x<D_xQ}
	||\mathcal{D}_x\phi||^2\lesssim ||\mathcal{D}_xQ_{m_N}\phi||^2+\alpha^2m_N^2||\phi||^2,
\end{align}
and
\begin{align}\label{eqn-D_xQ<D_x}
	||\mathcal{D}_xQ_{m_N}\phi||^2\lesssim ||\mathcal{D}_x\phi||^2+\alpha^2m_N^2||\phi||^2{,}
\end{align}
for all $\phi\in\mathcal{R}_N$.
\end{lemma}
{\bf Proof.}\quad Let us show \eqref{eqn-D_x<D_xQ} in detail only, and \eqref{eqn-D_xQ<D_x} can be obtained by the similar argument. Let $\phi=\sum_{k=0}^N\hat\phi_kH_k^\alpha(x)$ and $R_{m_N}=I-Q_{m_N}$, where $I$ is the identity operator{, then}
\begin{equation}\label{eqn-split Q}
	||\mathcal{D}_x\phi||^2\lesssim||\mathcal{D}_xQ_{m_N}\phi||^2+||\mathcal{D}_xR_{m_N}\phi||^2.
\end{equation}
We split $\phi$ in dyadic parts $\phi(x)=\sum_{k=0}^{m_N}\hat\phi_kH_k^\alpha(x)
+\sum_{j=1}^J\phi^j(x)$, where
\[
	\phi^j(x)=\sum_{k=2^{j-1}m_N+1}^{2^jm_N}\hat\phi_kH_k^\alpha(x),
\]
$j=1,\cdots,J$. Here{ $J={\lfloor\log_2\left(\frac N{m_N}\right)\rfloor}+1$} and $\hat\phi_k=0$ for $k=N+1,\cdots,2^Jm_N$. {The notation {$\lfloor\circ\rfloor$} means the largest integer less than or equal to $\circ$.} From the orthogonality relation \eqref{eqn-orthogonality of D_xH}, one has
\begin{equation}\label{eqn-bound two terms of D_xR}
	||\mathcal{D}_xR_{m_N}\phi||^2
	=\left|\left|\mathcal{D}_xR_{m_N}\sum_{k=0}^{m_N}\hat\phi_kH_k^\alpha\right|\right|^2+\sum_{j=1}^J||\mathcal{D}_xR_{m_N}\phi^j||^2.
\end{equation}
{ Recall that given a linear operator $R$} defined in $\mathcal{R}_N$ such that
\[
	R\left(\sum_{k=0}^N\hat\phi_kH_k^\alpha(x)\right)
	=\sum_{k=0}^N\hat{r}_k\hat\phi_kH_k^\alpha(x),
\]
where $\hat{r}_0,\cdots,\hat{r}_N$ are real numbers. Then for all $\phi\in\mathcal{R}_N$,
\begin{align}\label{eqn-R operator}
	||\mathcal{D}_xR\phi||^2
	\overset{\eqref{eqn-orthogonality of D_xH}}=\sum_{k=0}^N\hat{r}_k^2\hat\phi_k^2\lambda_k
	\leq\left(\sum_{k=0}^N\hat{r}_k^2\lambda_k\right)\left(\sum_{k=0}^N\hat\phi_k^2\right)
	=\left(\sum_{k=0}^N\hat{r}_k^2\lambda_k\right)||\phi||^2.
\end{align}
{We shall bound each summand on the right-hand side of \eqref{eqn-bound two terms of D_xR}. For the first summand, we have} 
\begin{equation}\label{eqn-estimate small}
	\left|\left|\mathcal{D}_xR_{m_N}\sum_{k=0}^{m_N}\hat\phi_kH_k^\alpha\right|\right|
	\overset{\eqref{eqn-R operator}}\leq\left(\sum_{k=0}^{m_N}(1-\hat q_k)^2\lambda_k\right)\left(\sum_{k=0}^{m_N}\hat\phi_k^2\right)
	\lesssim \alpha^2m_N^2\left(\sum_{k=0}^{m_N}\hat\phi_k^2\right)
	=\alpha^2m_N^2\left|\left|\sum_{k=0}^{m_N}\hat\phi_kH_k^\alpha\right|\right|^2{,}
\end{equation}
{since $\hat{q}_k=0$, for $k\leq m_N$;} while for the second summand, we obtain that, for any $j=1,\cdots,J$,
\begin{equation}\label{eqn-estimate large}
	||\mathcal{D}_xR_{m_N}\phi^j||^2
	\overset{\eqref{eqn-R operator}}\lesssim\left(\sum_{k=2^{j-1}m_N+1}^{2^jm_N}(1-\hat q_k)^2\lambda_k\right)||\phi^j||^2
	\lesssim\alpha^2m_N^2\sum_{k=2^{j-1}m_N+1}^{2^jm_N}\frac1k||\phi^j||^2
	\lesssim\alpha^2m_N^2||\phi^j||^2{,}
\end{equation}
{since $\hat q_k\geq1-\frac{m_N}k$.} {Combining} \eqref{eqn-estimate small} and \eqref{eqn-estimate large}, {it yields that}
\begin{align}\label{eqn-estimate derivative of R phi}
	\left|\left|\mathcal{D}_xR_{m_N}\phi\right|\right|^2
	\lesssim\alpha^2m_N^2\left(\left|\left|\sum_{k=0}^{m_N}\hat\phi_k^2H_k^\alpha\right|\right|^2+\sum_{j=1}^J||\phi^j||^2\right)
	\lesssim\alpha^2m_N^2||\phi||^2.
\end{align}
Substituting {\eqref{eqn-estimate derivative of R phi}} back to \eqref{eqn-split Q}, we get the desired result \eqref{eqn-D_x<D_xQ}. {Equation} \eqref{eqn-D_xQ<D_x} {follows similarly from} {\eqref{eqn-estimate derivative of R phi} and}
\[
	||\mathcal{D}_xQ_{m_N}\phi||^2\lesssim||\mathcal{D}_x\phi||^2+||\mathcal{D}_xR_{m_N}\phi||^2.
\]
\hfill{$\Box$}

The apriori estimates on the approximate solution $u_N$ are obtained in the following lemma. {The technical condition \eqref{eqn-condition on xu_N} is necessary if we {want some control} on $||\D_xQ_{m_N}u_N||_{L^2(0,T;L^2(\R)}$, instead of $||\p_xQ_{m_N}u_N||_{L^2(0,T;L^2(\R))}$. {Actually}, $||\p_xQ_{m_N}u_N||_{L^2(0,T;L^2(\R))}$ can be {estimated without condition \eqref{eqn-condition on xu_N}}, but the compensated compactness arguments {will not work with only the estimate on $||\p_xQ_{m_N}u_N||_{L^2(0,T;L^2(\R))}$.}
\begin{lemma}\label{lemma-apriori estimate}
	Let $f\in C^1(\mathbb{R})$, and there exists a primitive function $\bar{F}(x)$ of $xf'(x)$, i.e. $\bar{F}'(x)=xf'(x)$. Let $u_0\in L^2(\mathbb{R})$, $T>0$, $Q_{m_N}$ is given in \eqref{eqn-Q} {and} \eqref{eqn-q}, and $u_N:\,[0,T]\times\R\rightarrow\mathcal{R}_N$ {is} the solution of \eqref{eqn-spectral scheme}. {Assume} that
\begin{equation}\label{eqn-condition on xu_N}
	||xu_N||_{L^2(0,T;L^2(\R))}\lesssim N^\t,
\end{equation}
for some $\t>0$. Then
\begin{align}\label{eqn-estimate of D_tQ}
	||\mathcal{D}_xQ_{m_N}u_N||_{L^2(0,T;L^2(\mathbb{R}))}\lesssim\left\{\begin{aligned}
	\frac1{\sqrt{\e_N}},\quad&\textup{if}\ \frac1{\sqrt{\e_N}}\gg N^\t\\
	N^\t,\quad&\textup{if}\ \frac1{\sqrt{\e_N}}\ll N^\t
\end{aligned}\right.,
\end{align}
and
\begin{align}\label{eqn-estimate of u_N}
	{||u_N(\cdot,{T})||}\lesssim\left\{\begin{aligned}
	1,\quad&\textup{if}\ \frac1{\sqrt{\e_N}}\gg N^\t\\
	\sqrt{\e_N}N^\t,\quad&\textup{if}\ \frac1{\sqrt{{\e_N}}}\ll N^\t
\end{aligned}\right.,
\end{align}
where the generic constant contains in $\lesssim$ may depend on $\a,T$ etc., but not $N$.
\end{lemma}
{\bf Proof.}\quad Let us choose $\varphi=u_N\in\mathcal{R}_N$ in \eqref{eqn-spectral scheme} and it yields that
	\begin{align}\label{eqn-apriori0}
		0=&\int_{\mathbb{R}}u_N\partial_tu_Ndx+\int_{\mathbb{R}}\partial_x(P_{N+1}f(u_N))u_Ndx
-\epsilon_N\int_{\mathbb{R}}\p_x\D_x(Q_{m_N}u_N)u_Ndx.
	\end{align}
It is clear that the first term on the right-hand side of \eqref{eqn-apriori0} is $\frac12\frac d{dt}||u_N||^2$ and the second term is zero. In fact, the second term gives
\begin{align}\label{eqn-projection term}
	\int_{\mathbb{R}}\partial_x(P_{N+1}f(u_N))u_Ndx
	=&\int_{\mathbb{R}}P_{N+1}\partial_x(f(u_N))u_Ndx\\\notag
	&-\frac12\sqrt{\lambda_{N+2}}\int_{\mathbb{R}}\left[\widehat{f(u_N)}_{N+1}(t)H_{N+2}^{\alpha}(x)+\widehat{f(u_N)}_{N+2}(t)H_{N+1}^{\alpha}(x)\right]u_Ndx\\\notag
	=&\int_{\mathbb{R}}P_{N+1}\partial_x(f(u_N))u_Ndx
	=\int_{\mathbb{R}}\partial_x(f(u_N))u_Ndx
	=\int_{\R}f'(u_N)u_N\p_xu_Ndx,
\end{align}
where the first equality {in} \eqref{eqn-projection term} follows from the fact that
\begin{equation}\label{eqn-communication of P and derivative}
	 P_N\partial_x\phi(x,t)-\partial_xP_N\phi(x,t)=\frac12\sqrt{\lambda_{N+1}}\left[\hat{\phi}_N(t)H_{N+1}^{\alpha}(x)+\hat\phi_{N+1}(t)H_N^{\alpha}(x)\right],
\end{equation}
{and} the second and third equalities {in} \eqref{eqn-projection term} hold due to the orthogonality of generalized Hermite function. If there exists a primitive function $\bar{F}(x)$ of $xf'(x)$, with the fact that for any $N$, $\lim_{|x|\rightarrow\pm\infty}u_N(x)=0$, we obtain that
\begin{equation}\label{eqn-F=0}
	\left.\int f'(u_N)u_Ndu_N=\bar{F}(u_N(x))\right|_{x=\pm\infty}=0.
\end{equation}

Next, we shall examine the last term on the right-hand side of \eqref{eqn-apriori0}. {By} integration by parts, we have
\begin{align}\label{eqn-L_a fg}\notag
-\e_N\int_{\R}\p_x\D_xQ_{m_N}u_N\,u_Ndx&=\e_N\int_{\R}\D_xQ_{m_N}u_N\,\p_xu_Ndx\\\notag
	&=\e_N\int_{\R}\D_xQ_{m_N}u_N\,\D_xu_Ndx-\e_N\a^2\int_{\R}\D_xQ_{m_N}u_N\,(xu_N)dx\\
&=\I-\e_N\a^2\II.
\end{align}
Let us compute $\I$ and $\II$ term by term:
\begin{align}\label{eqn-I}\notag
	\I \overset{\eqref{eqn-D_xH_n}}=& \e_N\int_{\R}\left(\sum_{k=0}^N\hat{q}_k\hat{u}_k\sqrt{\l_k}H_{k-1}^\a\right)\left(\sum_{m=0}^N\hat{u}_m\sqrt{\l_m}H_{m-1}^\a\right)dx
\overset{\eqref{orthogonal}}=\e_N\sum_{k=0}^N\hat{q}_k\hat{u}_k^2\l_k\\
>&\e_N\sum_{k=0}^N\hat{q}^2_k\hat{u}_k^2\l_k=\e_N\int_{\R}|\D_xQ_{m_N}u_N|^2dx=\e_N||\D_xQ_{m_N}u_N||^2,
\end{align}
since $\hat{q}_k<1$, and $\II$ can be estimated as
\begin{align*}
	\II \leq&\frac1{2\g}||\D_xQ_{m_N}u_N||^2+\frac\g2||xu_N||^2,
\end{align*}
with $\g>\frac{\a^2}2$, by Young's inequality. Therefore, equation \eqref{eqn-apriori0} can be estimated as
\begin{align}\label{eqn-apriori1}
	0\geq\frac12\frac d{dt}||u_N||^2+\e_N\left(1-\frac{\a^2}{2\g}\right)||\D_xQ_{m_N}u_N||^2-\frac{\e_N\a^2\g}2||xu_N||^2.
\end{align}
Integrating {the both sides of \eqref{eqn-apriori1} with respect to $t$} from $0$ to $T$, we get
\begin{equation}\label{eqn-Lemma3.3}
	\e_N\a^2\g||xu_N||^2_{L^2(0,T;L^2(\R))}+||u_0||^2\geq||u_N||^2(T)+2\e_N\left(1-\frac{\a^2}{2\g}\right)||\D_xQ_{m_N}u_N||^2_{L^2(0,T;L^2(\R))}.
\end{equation}
Hence, \eqref{eqn-estimate of D_tQ} and \eqref{eqn-estimate of u_N} {follows} immediately {from \eqref{eqn-Lemma3.3}}.\hfill{$\Box$}

We are now ready to show {the convergence of the spectral scheme \eqref{eqn-spectral scheme} under some mild conditions.}
{\begin{theorem}\label{thm-u}
	Let $f\in C^1(\mathbb{R})$ be a nonlinear function such that $f(0)=0$, and there exists a primitive function $\bar{F}$ of $xf'(x)$, i.e., $\bar{F}'=xf'(x)$. Assume further that $u_0\in L^2(\mathbb{R})$. Let $u_N$ be the solution to the spectral approximation \eqref{eqn-spectral scheme}, which is uniformly bounded, i.e. 
{
\[
||u_N||_\infty<C,
\]
} independent of $N$, and
\begin{equation}\tag{3.10}
	||xu_N||_{L^2(0,T;L^2(\R))}\lesssim N^\t,
\end{equation}
for some $0<\t<\frac14$, holds. Let $N^{-\frac12}\ll\e_N\ll N^{-2\t}$, $m_N\ll N^\b$, with some $0<\b<\t$. Then $\{u_N\}$ converges (strongly in $L_{loc}^p(\Omega)$, $1\leq p<\infty$) to the unique entropy solution of the problem \eqref{eqn-nonlinear conservation}, {denoted as $u(x,t)$,} where $\Omega\in\mathbb{R}\times[0,T]$ is an open and bounded subset.
\end{theorem}}
\noindent{\bf Proof.}\quad The uniform boundedness of $\{u_N\}$ in $L^\infty(\mathbb{R}\times[0,T])$ guarantees that there exists a subsequence {converging} in the weak-* sense of $L^\infty$, {denoted} also $\{u_N\}$ and the limit $u$. We shall prove that $u$ is the unique entropy solution of \eqref{eqn-nonlinear conservation}, and the whole sequence $\{u_N\}$ tends to $u$ in $L_{loc}^p(\Omega)$, $1\leq p<\infty$.

 We first show that $\partial_t u_N+\partial_x f(u_N)$ is in a compact set of $H_{loc}^{-1}(\mathbb{R}\times(0,T))$.
\begin{align}\label{eqn-thm-0}
	\p_t u_N+\p_xf(u_N)=\e_N\p_x\D_xQ_{m_N}u_N+\p_x[(I-P_{N+1})f(u_N)].
\end{align}
Let $K$ be a compact set of $\R\times(0,T)$. It is obvious that the first term on the right-hand side of \eqref{eqn-thm-0} tends to $0$ in $H_{loc}^{-1}(\R\times(0,T))$, since
\begin{align}\label{eqn-DxQ}
	\e_N||\D_xQ_{m_N}u_N||_{L^2(K)}\overset{\eqref{eqn-estimate of D_tQ}}\lesssim \e_N\frac1{\sqrt{\e_N}}\rightarrow0.
\end{align}
According to  Lemma \ref{lemma-orthogonal error}, the second term on the right-hand side of \eqref{eqn-thm-0} can be estimated as
\begin{align}\label{eqn-(I-P)f-0}
	||(I-P_{N+1})f(u_N)||_{L^2(K)}\lesssim N^{-\frac12}||\D_xf(u_N)||_{L^2(0,T;L^2(\R))}{.}
\end{align}
Notice that
\begin{align}\label{eqn-Df}
	||\D_xf(u_N)||\leq||\p_xf(u_N)||+\a^2||xf(u_N)||\leq\sup_{|\xi|\leq||u_N||_\infty}|f'(\xi)|\left(||\p_xu_N||+\a^2||xu_N||\right),
\end{align}
where we use the fact that there exists {$\xi\in\mathbb{R}$}, such that $|\xi|\leq||u_N||_\infty$ and $f(u_N)=f'(\xi)u_N$, if $f(0)=0$ and $f\in C^1(\R)$. By Remark \ref{remark-equivalence of derivatives} and Lemma \ref{lemma-D_x<D_xQ}, we have
\begin{align}\label{eqn-xu_N}
	||\p_xu_N||\leq||\D_xu_N||+\a||u_N||\leq||\D_xQ_{m_N}u_N||+\a(m_N+1)||u_N||.
\end{align}
Thus, back to \eqref{eqn-(I-P)f-0}, we obtain that
\begin{align}\label{eqn-(I-P)f-1}
	||(I-P_{N+1})f(u_N)||_{L^2(K)}\lesssim N^{-\frac12}\left( \frac1{\sqrt{\e_N}}+m_N+N^\t\right)\ll \frac1{\sqrt{\e_NN}}+N^{-\frac12}m_N\rightarrow0,
\end{align}
since $N^{-\frac12}\ll\e_N\ll N^{-2\t}$ and $m_N\ll N^\b$, with $0<\b<\t<\frac14$. Therefore, we conclude that $\partial_t u_N+\partial_x f(u_N)$ is in a compact set of $H_{loc}^{-1}(\mathbb{R}\times(0,T))$.

Let $(U,F)$ be an entropy pair associated to \eqref{eqn-nonlinear conservation}. Next, we shall show that $\partial_t U(u_N)+\partial_x F(u_N)$ is also in a compact subset of $H^{-1}_{loc}(\mathbb{R}\times(0,T))$. {Let us compute directly:}
\begin{align}\label{eqn-U(uN)}\notag
	\p_tU(u_N)+\p_xF(u_N)
=&U'(u_N)(\partial_tu_N+\partial_xf(u_N))\\\notag
=&\e_NU'(u_N)\p_x\D_xQ_{m_N}u_N+U'(u_N)\p_x(I-P_{N+1})f(u_N)\\\notag
=&\e_N\p_x(U'(u_N)\D_xQ_{m_N}u_N)-\e_NU''(u_N)\p_xu_N\D_xQ_{m_N}u_N\\
&+\p_x(U'(u_N)(I-P_{N+1})f(u_N))-U''(u_N)\p_xu_N(I-P_{N+1})f(u_N).
\end{align}
The first and third term on the right-hand side of \eqref{eqn-U(uN)} can be estimated similarly as in \eqref{eqn-DxQ} and \eqref{eqn-(I-P)f-1}. Indeed, {we have}
\begin{align*}
	\e_N||U'(u_N)\D_xQ_{m_N}u_N||_{L^2(K)}\leq \e_N||U'(u_N)||_\infty||\D_xQ_{m_N}u_N||_{L^2(0,T;L^2(\R))}\overset{\eqref{eqn-estimate of D_tQ}}\lesssim \e_N\frac1{\sqrt{\e_N}}\rightarrow0,
\end{align*}
and
\begin{align*}
	||U'(u_N)(I-P_{N+1})f(u_N)||_{L^2(K)}\leq& ||U'(u_N)||_\infty||(I-P_{N+1})f(u_N)||_{L^2(0,T;L^2(\R))}\\
\overset{\eqref{eqn-(I-P)f-1}}\ll&\frac1{\sqrt{\e_NN}}+N^{-\frac12}m_N\rightarrow0{,}
\end{align*}
{where $||U'(u_N)||_\infty<\infty$, since $U\in C^2$ and $||u_N||_\infty<C$.} Therefore, $\e_N\p_x(U'(u_N)\D_xQ_{m_N}u_N)\rightarrow0$ and $\p
_x(U'(u_N)(I-P_{N+1})f(u_N))\rightarrow0$ in $H_{loc}^{-1}(\R\times(0,T))$. The second and fourth term {on} the right-hand side of \eqref{eqn-U(uN)} are estimated {below:}
\begin{align*}
	\e_N||U''(u_N)\p_xu_N&\D_xQ_{m_N}u_N||_{L^1(K)}\\
\leq&\e_N||U''(u_N)||_\infty||\p_xu_N||_{L^2(0,T;L^2(\R))}||\D_xQ_{m_N}u_N||_{L^2(0,T;L^2(\R))}\\
\overset{\eqref{eqn-xu_N}}\lesssim&\e_N||\D_xQ_{m_N}u_N||_{L^2(0,T;L^2(\R))}\left(||\D_xQ_{m_N}u_N||_{L^2(0,T;L^2(\R))}+m_N||u_N||_{L^2(0,T;L^2(\R))}\right)\\
\overset{\eqref{eqn-estimate of D_tQ},\eqref{eqn-estimate of u_N}}\lesssim& \e_N\frac1{\sqrt{\e_N}}\left(\frac1{\sqrt{\e_N}}+m_N\right)=\mathcal{O}(1),
\end{align*}
since $\sqrt{\e_N}m_N\ll N^{-\t+\b}\rightarrow0$, as $N\rightarrow\infty${, and $||U''(u_N)||_\infty<\infty$ ($U\in C^2$ and $||u_N||_\infty<C$)}{, and}
\begin{align}\label{eqn-fourth term}\notag
	||U''(u_N)&\p_xu_N(I-P_{N+1})f(u_N)||_{L^1(K)}\\\notag
\leq&||U''(u_N)||_\infty||\p_xu_N||_{L^2(0,T;L^2(\R))}||(I-P_{N+1})f(u_N)||_{L^2(0,T;L^2(\R))}\\
\overset{\eqref{eqn-xu_N},\eqref{eqn-(I-P)f-1}}\ll& \left(\frac1{\sqrt{\e_N}}+m_N\right)N^{-\frac12}\left(\frac1{\sqrt{\e_N}}+m_N\right)
\leq \frac1{\e_N\sqrt{N}}+\frac{m_N^2}{\sqrt{N}}\rightarrow0,
\end{align}
since {$||U''(u_N)||_\infty<\infty${,}} $\frac1{\e_N\sqrt{N}}\ll \frac1{N^{-\frac12}\sqrt{N}}=1$ and $\frac{m_N^2}{\sqrt{N}}\ll N^{2\b-\frac12}\rightarrow0$, with the assumption that $\b<\t<\frac14$.

Thus the entropy production $\partial_tU(u_N)+\partial_xF(u_N)$ can be written as a sum of four terms, two are bounded in $L^1(\Omega)$  and the other two tend to $0$ in $H^{-1}_{loc}(\Omega)$. Besides, $\partial_tU(u_N)+\partial_xF(u_N)$ is in $W^{-1,p}_{loc}(\mathbb{R}\times(0,T))$ for any $p>2$, since $U$ and $F$ are continuous and $u_N$ is uniformly bounded in $L^\infty(\mathbb{R}\times(0,T))$. Therefore, in view of the Murat's lemma \cite{Ch}, $\partial_tU(u_N)+\partial_xF(u_N)$ is in a compact subset of $H^{-1}_{loc}(\mathbb{R}\times(0,T))$.

We conclude that the entropy production of \eqref{eqn-spectral scheme} is $H^{-1}-$compact, by compensated compactness arguments \cite{T}{. It} implies that $u_N$ converges strongly in $L_{loc}^p(\Omega)$, $1\leq p<\infty$ to a weak solution of the conservation law \eqref{eqn-nonlinear conservation}. {Let us denote this weak solution $u$.}

It remains to show that $u$ is indeed the weak solution of \eqref{eqn-nonlinear conservation} satisfying the entropy condition. Let us multiply a nonnegative test function $\phi\in C_0^1(\R\times(0,T))$ on both sides of \eqref{eqn-U(uN)} and integrate it with respect to both $t$ and $x$:
\begin{align}\label{eqn-entropy}\notag
	\int_0^T\int_{\R}&\left[\p_tU(u_N)+\p_xF(u_N)\right]\phi dxdt\\\notag
\overset{\eqref{eqn-U(uN)}}=&-\e_N\int_0^T\int_\R U'(u_N)\D_xQ_{m_N}u_N\p_x\phi dxdt
-\e_N\int_0^T\int_RU''(u_N)\p_xu_N\D_xQ_{m_N}u_N\phi dxdt\\
&-\int_0^T\int_\R U'(u_N)(I-P_{N+1})f(u_N)\p_x\phi dxdt-\int_0^T\int_\R U''(u_N)\p_xu_N(I-P_{N+1})f(u_N)\phi dxdt.
\end{align}
The first, third and fourth term on the right-hand side of \eqref{eqn-entropy} tend to $0$, as $N\rightarrow\infty$. {It is because that}
\begin{align}\label{eqn-thm-1}\notag
	&\e_N\left|\int_0^T\int_\R U'(u_N)\D_xQ_{m_N}u_N\p_x\phi dxdt\right|\\
&\qquad\leq \e_N||U'(u_N)||_\infty||\D_xQ_{m_N}u_N||_{L^2(0,T;L^2(\R))}||\p_x\phi||_{L^2(0,T;L^2(\R))}
\lesssim \e_N\frac1{\sqrt{\e_N}}\rightarrow0{,}
\end{align}
\begin{align}\label{eqn-thm-3}\notag
	&\left|\int_0^T\int_\R U'(u_N)(I-P_{N+1})f(u_N)\p_x\phi dxdt\right|\\
&\quad\leq||U'(u_N)||_\infty||(I-P_{N+1})f(u_N)||_{L^2(0,T;L^2(\R))}||\p_x\phi||_{L^2(0,T;L^2(\R))}\overset{\eqref{eqn-(I-P)f-1}}\ll \frac1{\sqrt{\e_NN}}+N^{-\frac12}m_N\rightarrow0{,}
\end{align}
and
\begin{align}\label{eqn-thm-4}\notag
	&\left|\int_0^T\int_\R U''(u_N)\p_xu_N(I-P_{N+1})f(u_N)\phi dxdt\right|\\\notag
&\qquad\leq||U''(u_N)||_\infty||\p_xu_N||_{L^2(0,T;L^2(\R))}||(I-P_{N+1})f(u_N)||_{L^2(0,T;L^2(\R))}||\phi||_\infty\\
&\qquad\overset{\eqref{eqn-fourth term}}\lesssim\frac1{\e_N\sqrt{N}}+\frac{m_N^2}{\sqrt{N}}\rightarrow0.
\end{align}
{The} third term on the right-hand side of \eqref{eqn-entropy} {is analyzed below}. Notice that $Q_{m_N}=I-R_{m_N}$, then
\begin{align}\label{eqn-third term}\notag
	-\e_N&\int_0^T\int_RU''(u_N)\p_xu_N\D_xQ_{m_N}u_N\phi dxdt\\\notag
=&-\e_N\int_0^T\int_\R U''(u_N)(\D_xu_N)^2\phi dxdt+\e_N\int_0^T\int_RU''(u_N)\D_xu_N\D_xR_{m_N}u_N\phi dxdt\\
&+\e_N\a^2\int_0^T\int_R U''(u_N)xu_N\D_xQ_{m_N}u_N\phi dxdt.
\end{align}
The second and third term on the right-hand side of \eqref{eqn-third term} tend to $0$, as $N\rightarrow\infty$. In fact, {it is clear to see that}
\begin{align*}
	\e_N&\left|\int_0^T\int_RU''(u_N)\D_xu_N\D_xR_{m_N}u_N\phi dxdt\right|\\
&\qquad\leq\e_N||U''(u_N)||_\infty||\D_xu_N||_{L^2(0,T;L^2(\R))}||\D_xR_{m_N}u_N||_{L^2(0,T;L^2(\R))}||\phi||_\infty\\
&\qquad\overset{\eqref{eqn-D_x<D_xQ},\eqref{eqn-estimate derivative of R phi}}\lesssim \e_N\left(\frac1{\sqrt{\e_N}}+m_N\right)m_N\rightarrow0,
\end{align*}
and
\begin{align*}
	\e_N&\a^2\left|\int_0^T\int_R U''(u_N)xu_N\D_xQ_{m_N}u_N\phi dxdt\right|\\
\leq&\e_N\a^2||U''(u_N)||_\infty||xu_N||_{L^2(0,T;L^2(\R))}||\D_xQ_{m_N}u_N||_{L^2(0,T;L^2(\R))}||\phi||_\infty
\overset{\eqref{eqn-condition on xu_N},\eqref{eqn-estimate of D_tQ}}\lesssim \e_NN^\t\frac1{\sqrt{\e_N}}\rightarrow0.
\end{align*}
{Due} to the convexity of $U$, the first term on the right-hand side of \eqref{eqn-third term} is nonpositive. Therefore, as $N\rightarrow\infty$, the second term on the right-hand side of \eqref{eqn-entropy} is nonpositive. Combining \eqref{eqn-thm-1}-\eqref{eqn-thm-4}, we conclude that for any nonnegative test function $\phi\in C_0^1(\R\times(0,T))$,
\[
	\lim_{N\rightarrow\infty}\int_0^T\int_\R\left[\p_tU(u_N)+\p_xF(u_N)\right]\phi dxdt\leq0.
\]
{This reveals that} the entropy condition \eqref{eqn-entropy condition} has been satisfied in the weak sense.\hfill{$\Box$}
\begin{remark}
	The conditions on $\e_N$ and $m_N$ {in Theorem \ref{thm-u}} are almost the same as those in \cite{AR}. The difference is that we replace the condition $||xu_N||_\infty<C$, by some growth condition on $||xu_N||_{L^2(0,T;L^2(\R))}$ (\eqref{eqn-condition on xu_N} with $\t<\frac14$). It is hard to tell which condition is more restrictive.
\end{remark}

\subsection{With viscosity term $\e\L_\a u$}
In {this} subsection, we introduce another viscosity term $\e\L_\a u$. Unlike the viscosity operator $Q_{m_N}$ only modified the high frequency {modes}, this viscosity includes all. The spectral {scheme with this viscosity} is introduced in \eqref{eqn-spectral scheme-v}. Let us start with the apriori  estimates on {the approximate solution} $v_N$.
\begin{lemma}\label{lemma-apriori estimate-v}
	Let $f\in C^1(\mathbb{R})$, and there exists a primitive function $\bar{F}(x)$ of $xf'(x)$, i.e. $\bar{F}'(x)=xf'(x)$. Let $u_0\in L^2(\mathbb{R})$, $T>0$  and $v_N:\,[0,T]\times\R\rightarrow\mathcal{R}_N$ the solution of \eqref{eqn-spectral scheme-v}. Then
\begin{align}\label{eqn-estimate of D_xv}
	||\mathcal{D}_xv_N||_{L^2(0,T;L^2(\mathbb{R}))}\lesssim\frac1{\sqrt{\e_N}},
\end{align}
and
\begin{align}\label{eqn-estimate of v_N}
	{||v_N(\cdot,{T})||}\leq ||u_0||.
\end{align}
\end{lemma}
{\bf Proof.}\quad We multiply \eqref{eqn-spectral scheme-v} by $\varphi=v_N\in\mathcal{R}_N$ and integrate {it} with respect to $x$:
\begin{equation}\label{eqn-energy estimate-v}
	0=\int_\R(\p_tv_N)v_Ndx+\int_\R\p_x(f(v_N))v_Ndx+\e_N\int_\R(\L_\a v_N)v_Ndx=\frac12\frac{d}{dt}||v_N||^2+\e_N||\D_xv_N||^2,
\end{equation}
where the second term {in the middle of \eqref{eqn-energy estimate-v} vanishes due to the same reason in \eqref{eqn-F=0}, and the second term on the right-hand side of \eqref{eqn-energy estimate-v}} is followed from the fact that
\begin{align*}
	\int_\R(\L_\a\psi)\phi dx\overset{{\eqref{S-L}}}=&-\int_\R e^{\frac12\alpha^2x^2}{\frac d{dx}}\left(e^{-\alpha^2x^2}{\frac d{dx}}\left(e^{\frac12\alpha^2x^2}\psi\right)\right)\phi dx
=-\int_\R\left(\left({\frac d{dx}}-\a^2 x\right)\D_x\psi\right)\phi dx\\
=&\int_\R\D_x\psi\D_x\phi dx.
\end{align*}
Integrating on both sides of \eqref{eqn-energy estimate-v} from $0$ to $T$, we obtain that
\[
	||u_0||^2{=}||v_N||^2(T)+2\e_N||\D_xv_N||^2_{L^2(0,T;L^2(\R))}.
\]
Equation \eqref{eqn-estimate of D_xv} and \eqref{eqn-estimate of v_N} {are obtained} immediately.\hfill{$\Box$}

{We are now in the position to show the convergence of the scheme \eqref{eqn-spectral scheme-v}.} {The proof of the convergence of the scheme \eqref{eqn-spectral scheme-v} is similar to that of Theorem \ref{thm-u}. The differences are the delicate estimates, like those in \eqref{eqn-DxQ}-\eqref{eqn-Df}, \eqref{eqn-fourth term}-\eqref{eqn-third term}, etc.}
{
\begin{theorem}\label{thm-v}
	Let $f\in C^1(\mathbb{R})$ be a nonlinear function such that $f(0)=0$, and there exists a primitive function $\bar{F}$ of $xf'(x)$. Assume further that $u_0\in L^2(\mathbb{R})$. Let $v_N$ be the solution to the spectral approximation \eqref{eqn-spectral scheme-v}, which is uniformly bounded{, i.e.
\[
	||v_N||_\infty<C,
\]
independent of $N$, and} assume that 
\begin{equation}\label{cond-x2v_N}
	||x^2v_N||_{L^1(\R\times(0,T))}\ll\frac1{\e_N}.
\end{equation}
Let $\frac1{\e_N\sqrt{N}}\rightarrow0$. Then $\{v_N\}$ converges strongly in $L_{loc}^p(\Omega)$, $1\leq p<\infty$ to the unique entropy solution of the problem \eqref{eqn-nonlinear conservation}, where $\Omega\in\mathbb{R}\times[0,T]$ is an open and bounded subset.
\end{theorem}}
\noindent{\bf Proof.}\quad The uniform boundedness of $\{v_N\}$ in $L^\infty(\mathbb{R}\times[0,T])$ guarantees that there exists a subsequence {converging} in the weak-* sense of $L^\infty$, {denoted also as} $\{v_N\}$ and the limit $u$. We shall prove that $u$ is the unique entropy solution of \eqref{eqn-nonlinear conservation}, and the whole sequence $\{v_N\}$ tends to $u$ in $L_{loc}^p(\Omega)$, $1\leq p<\infty$.

 We first show that $\partial_t u_N+\partial_x f(u_N)$ is in a compact set of $H_{loc}^{-1}(\mathbb{R}\times(0,T))$. {Let us compute directly:}
\begin{align}\label{eqn-thm-v-0}\notag
	\p_t v_N+\p_xf(v_N)=&-\e_N\L_\a v_N+\p_x[(I-P_{N+1})f(v_N)]\\\notag
=&\e_N\p_x\D_xv_N-\e_N\a^2\p_x(xv_N)+\e_N\a^2v_N-\e_N\a^4x^2v_N+\p_x[(I-P_{N+1})f(v_N)]\\
=&\e_N\p_x^2v_N+\e_N\a^2v_N-\e_N\a^4x^2v_N+\p_x[(I-P_{N+1})f(v_N)].
\end{align}
Let $K\subset\R\times(0,T)$ be a compact set. Notice that
\begin{align}\label{eqn-thm-v-pv}
	\e_N||\p_xv_N||_{L^2(K)}
\overset{\eqref{eqn-equivalence of d_x and D_x}}\leq&\e_N(||\D_xv_N||_{L^2(0,T;L^2(\R))}+\a||v_N||_{L^2(0,T;L^2(\R))})
\overset{\eqref{eqn-estimate of D_xv},\eqref{eqn-estimate of v_N}}\lesssim\e_N\frac1{\sqrt{\e_N}}\rightarrow0{,}\\\label{eqn-thm-v-vN}
	\e_N||v_N||_{L^2(K)}\lesssim& \e_N\rightarrow0{,}\\\label{eqn-thm-v-I-PN}\notag
	||(I-P_{N+1})f(v_N)||_{L^2(K)}\lesssim& N^{-\frac12}||\D_xf(v_N)||_{L^2(0,T;L^2(\R))}\\\notag
\overset{\eqref{eqn-Df}}\lesssim&N^{-\frac12}\left(||\p_xv_N||_{L^2(0,T;L^2(\R))}+\a^2||xv_N||_{L^2(0,T;L^2(\R))}\right)\\
\overset{\eqref{eqn-equivalence of d_x and D_x},\eqref{eqn-equivalence of xphi}}\lesssim& N^{-\frac12}\frac1{\sqrt{\e_N}}\rightarrow0{,}
\end{align}
and
\begin{align}\label{eqn-thm-v-x2vN}
	\e_N||x^2v_N||_{L^1(K)}\overset{{\eqref{cond-x2v_N}}}\ll\e_N\frac1{\e_N}=1.
\end{align}
Thus, $\partial_tv_N+\partial_xf(v_N)$ can be written as a sum of four terms, two tend to $0$ in $H^{-1}_{loc}(\R)$, one is bounded in $L^1_{loc}(\R)$  and the other one tends to $0$ in $L^2_{loc}(\R)$ . Besides, $\partial_tv_N+\partial_xf(v_N)$ is in $W^{-1,p}_{loc}(\mathbb{R}\times(0,T))$ for any $p>2$, since $f\in C^2$ and {$v_N$} is uniformly bounded in $L^\infty(\mathbb{R}\times(0,T))$. Therefore, in view of the Murat's lemma \cite{Ch}, $\partial_tv_N+\partial_xf(v_N)$ is in a compact subset of $H^{-1}_{loc}(\mathbb{R}\times(0,T))$.

Next, we show that $\p_tU(v_N)+\p_xF(v_N)$ is also in a compact subset of $H^{-1}_{loc}(\R\times(0,T))$, where $(U,F)$ is the entropy pair introduced in \eqref{eqn-entropy condition}.
\begin{align}\label{eqn-thm-v-U}\notag
	\p_tU(v_N)+\p_xF(v_N)=&-\e_NU'(v_N)\L_\a v_N+U'(v_N)\p_x(I-P_{N+1})f(v_N)\\\notag
	\overset{\eqref{eqn-thm-v-0}}=&\e_N\p_x(U'(v_N)\p_xv_N)-\e_NU''(v_N)(\p_xv_N)^2\\\notag
	&+\e_NU'(v_N)\a^2v_N-\e_NU'(v_N)\a^4x^2v_N\\
	&+\p_x(U'(v_N)\p_x(I-P_{N+1})f(v_N))-U''(v_N)\p_xv_N(I-P_{N+1})f(v_N).
\end{align}
Notice the estimates in \eqref{eqn-thm-v-pv}-\eqref{eqn-thm-v-x2vN} and the fact that $U\in C^2$, $||v_N||_\infty<C$, the first, third, fourth and fifth term on the right-hand side of \eqref{eqn-thm-v-U} can be {dealt with} similarly as before, i.e. the first and fifth term tend to $0$ in $H^{-1}_{loc}(\R\times(0,T))${,} the third term tends to $0$ in $L^2_{loc}(\R\times(0,T))${,} and the fourth term is bounded in $L^1_{loc}(\R\times(0,T))$. The two {remaining} terms on the right-hand side of \eqref{eqn-thm-v-U} are both bounded in $L^1_{loc}(\R\times(0,T))$. In fact, {we have}
\begin{align*}
	\e_N||U''(v_N)(\p_xv_N)^2||_{L^1(K)}\leq\e_N||U''(v_N)||_\infty||\p_xv_N||_{L^2(0,T;L^2(\R))}^2\lesssim \e_N\left(\frac1{\sqrt{\e_N}}\right)^2=1,
\end{align*}
and
\begin{align}\label{eqn-thm-v-U''pv}\notag
	||U''(v_N)&\p_xv_N(I-P_{N+1})f(v_N)||_{L^1(K)}\\\notag
\leq&||U''(v_N)||_\infty||\p_xv_N||_{L^2(0,T;L^2(\R))}||(I-P_{N+1})f(v_N)||_{L^2(0,T;L^2(\R))}\\
\overset{\eqref{eqn-thm-v-pv},\eqref{eqn-thm-v-I-PN}}\lesssim&\frac1{\sqrt{\e_N}}N^{-\frac12}\frac1{\sqrt{\e_N}}=\frac1{\e_N\sqrt{N}}\rightarrow0.
\end{align}
Therefore, as we argued before, in view of the Murat's lemma \cite{Ch}, $\partial_tU(v_N)+\partial_xF(v_N)$ is in a compact subset of $H^{-1}_{loc}(\mathbb{R}\times(0,T))$. We conclude that $v_N$ converges strongly in $L_{loc}^p(\R\times(0,T))$, $1\leq p<\infty$ to a weak solution of the conservation law \eqref{eqn-nonlinear conservation}. {Let us denote this solution as $u$.}

It remains to show that the entropy condition \eqref{eqn-entropy condition} is satisfied {by $u$} in the weak sense. Let us multiply a nonnegative test function $\phi\in C_0^1(\R\times(0,T))$ on both sides of \eqref{eqn-thm-v-U} and integrate it with respect to both $t$ and $x$:
\begin{align}\label{eqn-thm-v-entropy}\notag
	\int_0^T\int_\R&[\p_tU(v_N)+\p_xF(v_N)]\phi dxdt\\\notag
\overset{\eqref{eqn-thm-v-U}}=&-\e_N\int_0^T\int_\R U'(v_N)\p_xv_N\p_x\phi dxdt-\e_N\int_0^T\int_\R U''(v_N)(\p_xv_N)^2\phi dxdt\\\notag
&+\e_N\int_0^T\int_\R U'(v_N)\a^2v_N\phi dxdt-\e_N\int_0^T\int_\R U'(v_N)\a^4x^2v_N\phi dxdt\\
	&-\int_0^T\int_\R U'(v_N)\p_x(I-P_{N+1})f(v_N)\p_x\phi dxdt-\int_0^T\int_\R U''(v_N)\p_xv_N(I-P_{N+1})f(v_N)\phi dxdt.
\end{align}
{The estimates} \eqref{eqn-thm-v-pv}-\eqref{eqn-thm-v-x2vN} and \eqref{eqn-thm-v-U''pv} imply that all the terms except the second one on the right-hand side of \eqref{eqn-thm-v-entropy} tends to $0$, as $N\rightarrow\infty$. {It is hard to tell that} the second term is nonpositive, due to the convexity of $U$. Therefore, the entropy condition is satisfied in the weak sense, i.e. for any nonnegative test function $\phi\in C_0^1(\R\times(0,T))$, {we have} 
\[
	\lim_{N\rightarrow\infty}\int_0^T\int_\R[\p_tU(v_N)+\p_xF(v_N)]\phi dxdt\leq0.
\]
\hfill{$\Box$}
\begin{remark}
{Compared} the viscosity term in \eqref{eqn-spectral scheme} {with} \eqref{eqn-spectral scheme-v}, the convergence analysis for \eqref{eqn-spectral scheme-v} is easier, but the price to pay is {that} condition \eqref{cond-x2v_N} is stronger than \eqref{eqn-condition on xu_N}, since 
\[
	||xv_N||_{L^2(\R\times(0,T))}^2\leq||v_N||_\infty||x^2v_N||_{L^1(\R\times(0,T))}\ll\frac1{\e_N}{\ll N^{\frac12},}
\]
{ where $\frac1{\epsilon_N\sqrt{N}}\rightarrow0$.}
\end{remark}
\section{Numerical experiments}
\setcounter{equation}{0}

In this section, we use the spectral viscosity methods \eqref{eqn-spectral scheme} {and} \eqref{eqn-spectral scheme-v} to numerically solve the inviscid Burger's equation
\begin{equation}\label{eqn-Burger}
	\partial_tu+\frac12\partial_x(u^2)=0,
\end{equation}
in $\mathbb{R}$, with the initial condition ${u_0(x)}=e^{-x^2}$. We {shall solve} the same problem in \cite{AR} for the purpose of comparison. The exact solution is given implicitly by the method of characteristics, i.e.,
\begin{equation}\label{eqn-exact solution to BE}
	u(\eta+te^{-\eta^2},t)=e^{-\eta^2},
\end{equation}
with the initial condition $u(\eta,0)=u_0$. {The} shock presents at time $T^*=\left(\frac e2\right)^{\frac12}\approx1.1658$. All of the numerical results displayed below are at time $t=1.5>T^*$.

{In the sequel, in} both spectral schemes  \eqref{eqn-spectral scheme} {and} \eqref{eqn-spectral scheme-v}, we let $\varphi=H_m^\alpha(x)$, $m=0,1,\cdots,N$. The coefficients  {$\hat{u}_m(t)$ and $\hat{v}_m(t)$}, $m=0,\ldots,N$, are the solutions {to the corresponding} system of {nonlinear} ordinary differential equation. It is solved by using the fourth order Runge-Kutta method with adaptive time {steps} ({\it ode45} in Matlab).

The viscosity {operator  $Q_{m_N}$} in scheme \eqref{eqn-spectral scheme} { is defined by $\hat{q}_k$. We shall} try the following multipliers in \cite{AR}:
	{
\begin{align}\label{eqn-qk}\notag
	\hat{q}^1_k=&\frac N{N-m_N}\left(1-\frac{m_N}k\right),\\
	\hat{q}^2_k=&\frac{k-m_N}{N-m_N},\\\notag
	\hat{q}^3_k=&\exp{\left\{-\left(\frac{k-N}{k-m_N}\right)^2\right\}},
\end{align}}
for $k>m_N$. It is easy to check that condition \eqref{eqn-q} are satisfied {by $\hat{q}_k^1$. It is suggested in \cite{MT} that $\hat{q}_k^2$ and $\hat{q}_k^3$ may yield better resolution of the shock. However, the lower bound in \eqref{eqn-q} does not hold.}

{\subsection{The choice of scaling factor $\a$}}

{ In our numerical simulations, we introduce the generalized Hermite functions $H_k^\a$ \eqref{new hermite} with one more parameter $\a$ to tuning with.}  The optimal choice of the scaling factor to accurately resolve the functions is still open, let alone the solution to some partial differential equations. But the suitable choice of the scaling factor {to resolve certain kind of analytic/smooth functions} is investigated in \cite{Tang}, \cite{B}, \cite{B1}, \cite{LY}, etc. It is known so far that the scaling factor should match the asymptotical behavior of the function to be resolved. {The author and her co-worker provide a practical guideline to choose the suitable scaling factor \cite{LY} for Gaussian and super-Gaussian functions. The time-dependent scaling factor of the Hermite spectral method in solving evolution equations has also been investigated in \cite{LYY}. However, all the guidelines can not be applied in our case, due to the discontinuity.

The approximate solution $u_N$ of scheme \eqref{eqn-spectral scheme} with $\epsilon_N=0$, $N=30$ at time $T=1.5$ are plotted in Figure \ref{fig-withoutviscosity_differentalpha} with $\alpha$ varying from $0.5$, $1$, $\sqrt{2}$ and $3$. It reveals that the larger $\alpha$ gives better resolution of the discontinuity, but more oscillations. It is clearly shown in Figure \ref{fig-withoutviscosity_differentalpha} ($\a=3$) that without the help of viscosity the approximate solution does not converge to the entropy solution. Compared with Figure 6.1 in \cite{AR}, our scheme with $N=30$ can resolve the solution as good as the scheme in \cite{AR} with $N=257$. In Figure \ref{fig-withviscosity_q1_differentalpha}-\ref{fig-withviscosity_q3_differentalpha}, we experiment our scheme \eqref{eqn-spectral scheme} with $\hat{q}_k^1-\hat{q}_k^3$ in \eqref{eqn-qk}, $\epsilon_N=0.5N^{-0.33}$, $N=30$ and $\a=0.5$, $1$, $\sqrt{2}$, $2$.  They all show the similar phenomenon as that without viscosity that the larger $\a$ is, the better resolution at the discontinuity we obtain, the more oscillations the approximate solution presents. Obviously, with the same $N$, properly tuning  the scaling factor $\a$ can help the resolution of the discontinuity. It is not hard to see that from the definition of the generalized Hermite function \eqref{new hermite}, the larger $\a$ is, the more concentrated the generalized Hermite functions present. This is the possible reason why the larger $\a$ can resolve the discontinuity better.  Compared Figure \ref{fig-withoutviscosity_differentalpha}-\ref{fig-withviscosity_q3_differentalpha} with Figure \ref{fig-withoutviscosity}-\ref{fig-withviscosity_q3}, the properly choice of the scaling factor $\a$ can reduce $N$ significantly, so does the computational cost.  Figure \ref{fig-withviscosity_differentalpha} displays the approximate solution obtained by scheme \eqref{eqn-spectral scheme-v} with $\e_N=0.05N^{-0.33}$, $N=30$ and $\a=0.5,1,\sqrt{2},2$. Not like the phenomenon in Figure \ref{fig-withoutviscosity_differentalpha}-\ref{fig-withviscosity_q3_differentalpha}, Figure \ref{fig-withviscosity_differentalpha} shows that in our scheme \eqref{eqn-spectral scheme-v} large $\a$ (say $\a=2$) tends to smoothing out everything, including the discontinuity. It seems that even the energy has been dissipated due to the excessive viscosity term in Figure \ref{fig-withviscosity_differentalpha} ($\a=2$). From this numerical experiment, we believe that, besides the concentration, the larger $\a$ also introduces more viscosity. Therefore, the balance of concentration and dissipation should be reached to obtain the ideal resolution. One may wonder why the smoothing-out effect of large $\a$ in Figure  \ref{fig-withoutviscosity_differentalpha}-\ref{fig-withviscosity_q3_differentalpha}($\a=2$) is not as obvious as that in Figure \ref{fig-withviscosity_differentalpha}($\a=2$). Notice that the major difference of scheme \eqref{eqn-spectral scheme-v} and \eqref{eqn-spectral scheme} is that one modifies all modes of the Fourier-Hermite expansion, while the other one only modifies the high modes. Therefore, we believe it is the excessive modifications of the low modes that causes the over-smoothing in Figure \ref{fig-withviscosity_differentalpha}, but not in Figure \ref{fig-withoutviscosity_differentalpha}-\ref{fig-withviscosity_q3_differentalpha}. How to choose optimal scaling factor $\a$ is still open.}

\begin{figure}[!h]
    \includegraphics[trim = 20mm 95mm 0mm 95mm, clip,scale=0.9]{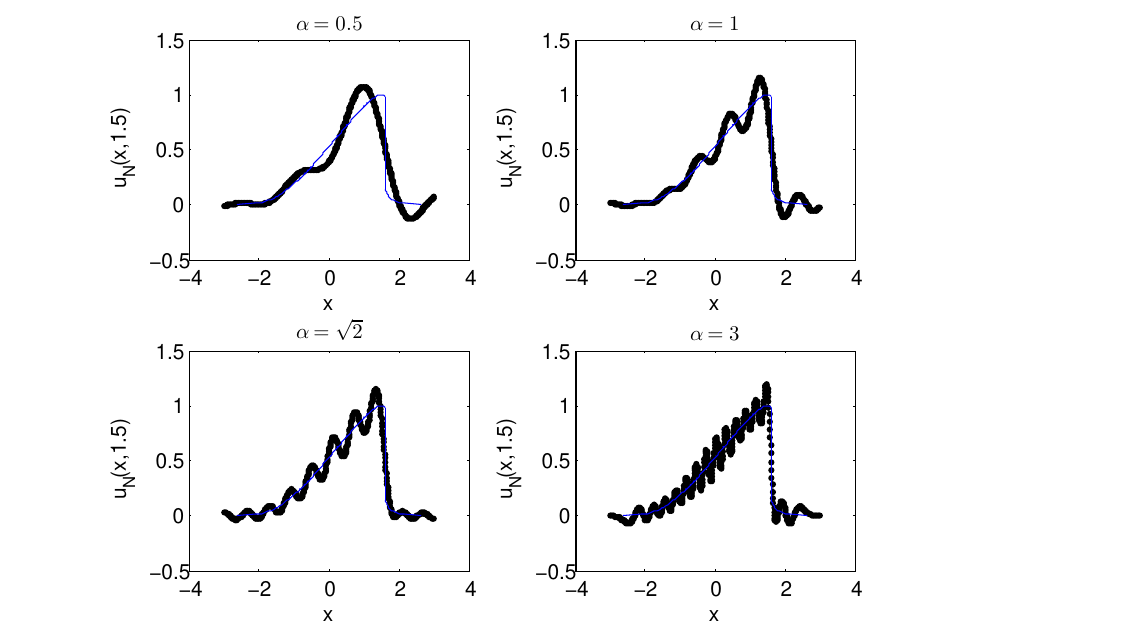}
		\caption{Solid blue line: the exact solution of Burger's equation; Dotted line: the scheme \eqref{eqn-spectral scheme} with $\epsilon_N=0$ (without viscosity) and $N=30$.}
\label{fig-withoutviscosity_differentalpha}
 \end{figure}

\begin{figure}[!h]
    \includegraphics[trim = 20mm 95mm 0mm 95mm, clip,scale=0.9]{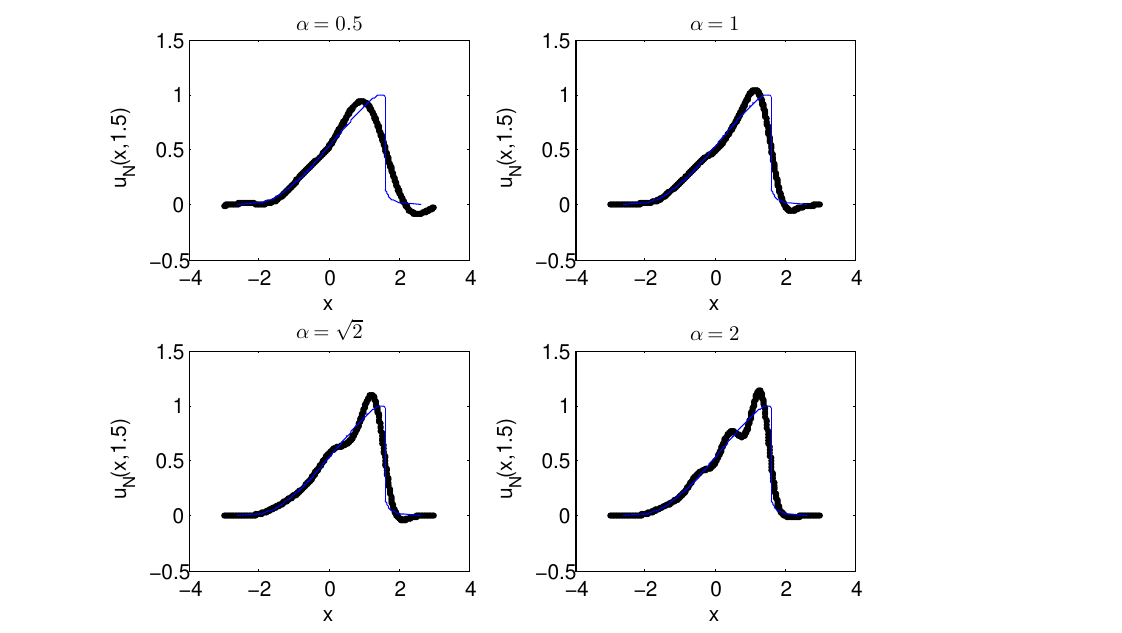}
		\caption{Solid blue line: the exact solution of Burger's equation; dotted line: the scheme \eqref{eqn-spectral scheme} with $\epsilon_N=0.5N^{-0.33}$ and $N=30$.}
\label{fig-withviscosity_q1_differentalpha}
 \end{figure}

\begin{figure}[!h]
    \includegraphics[trim = 20mm 95mm 0mm 95mm, clip,scale=0.9]{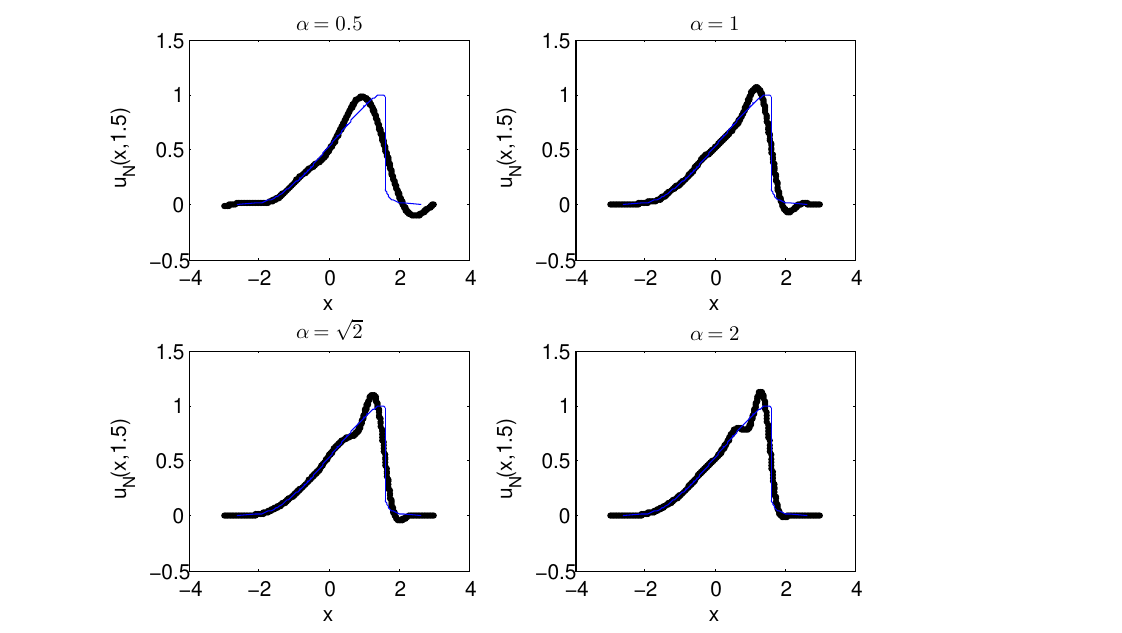}
		\caption{Solid blue line: the exact solution of Burger's equation; dotted line: the scheme \eqref{eqn-spectral scheme} with $\epsilon_N=0.5N^{-0.33}$ and $N=30$.}
\label{fig-withviscosity_q2_differentalpha}
 \end{figure}

\begin{figure}[!h]
    \includegraphics[trim = 20mm 95mm 0mm 95mm, clip,scale=0.9]{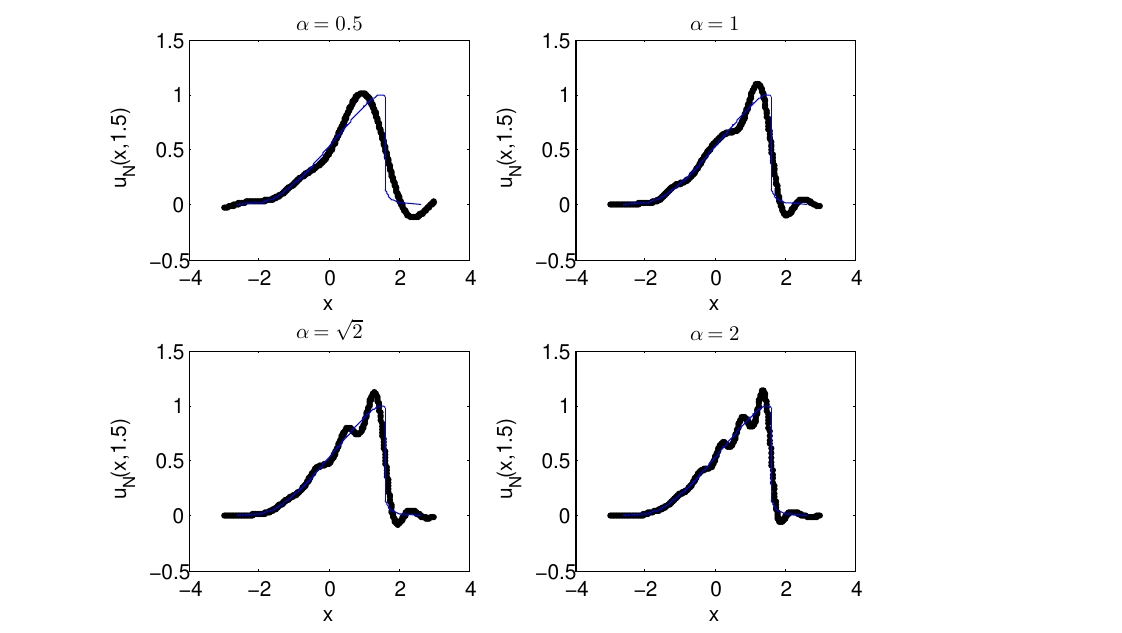}
		\caption{Solid blue line: the exact solution of Burger's equation; dotted line: the scheme \eqref{eqn-spectral scheme} with $\epsilon_N=0.5N^{-0.33}$ and $N=30$.}
\label{fig-withviscosity_q3_differentalpha}
 \end{figure}

\begin{figure}[!h]
    \includegraphics[trim = 20mm 95mm 0mm 95mm, clip,scale=0.9]{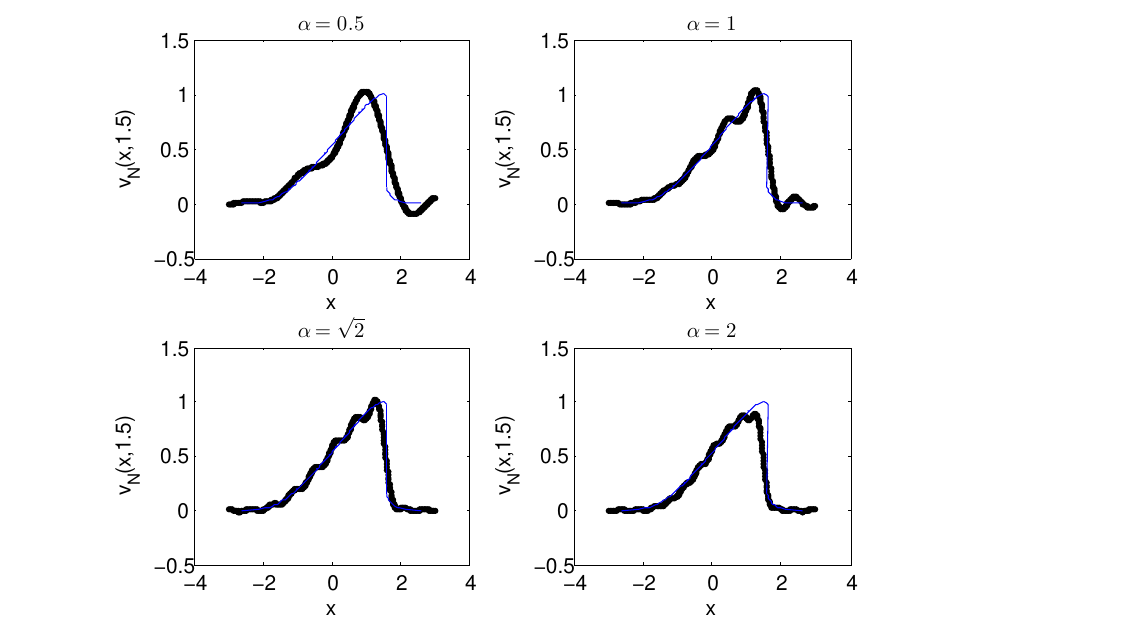}
		\caption{Solid blue line: the exact solution of Burger's equation; dotted line: the scheme \eqref{eqn-spectral scheme-v} with $\epsilon_N=0.05N^{-0.33}$ and $N=30$.}
\label{fig-withviscosity_differentalpha}
 \end{figure}

\begin{figure}[!h]
    \includegraphics[trim = 8mm 110mm 0mm 110mm, clip,scale=0.75]{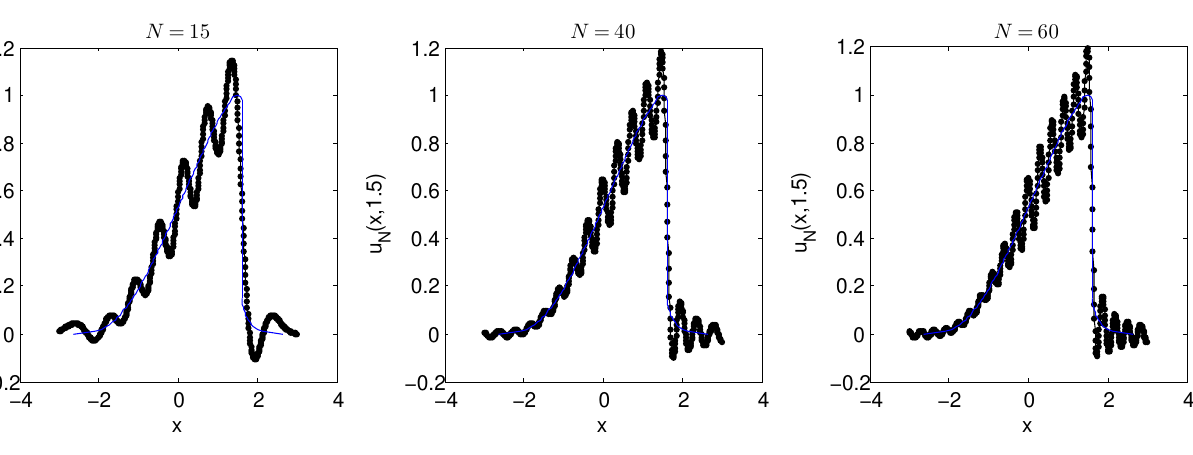}
		\caption{Solid blue line: the exact solution of Burger's equation; {dotted line: the approximate solution of the spectral scheme \eqref{eqn-spectral scheme} with $\a=2$, $\epsilon_N=0$,} $N=15,40$ and $60$, respectively.}
\label{fig-withoutviscosity}
 \end{figure}

{\subsection{Experiments with various $N$}}

In Figure \ref{fig-withoutviscosity}, we show the result of the spectral approximation without viscosity{, i.e. scheme \eqref{eqn-spectral scheme} with $\epsilon_N=0$, $\a=2$,} for $N=15,40$ and $60$, respectively. The larger $N$ is, the better resolution at the point of discontinuity we achieve, but the oscillations do not disappear  {as $N$ increases}. The Gibb's phenomenon prevents the convergence, even in the intervals where the exact solution is actually smooth. Compared with the pseudospectral viscosity method  in \cite{AR}{ (cf. Figure 6.1), the discontinuity can be resolved by our scheme} better with much smaller $N${, with the help of the scaling factor.}

\begin{figure}[!h]
	\includegraphics[trim = 8mm 115mm 0mm 115mm, clip,scale=0.75]{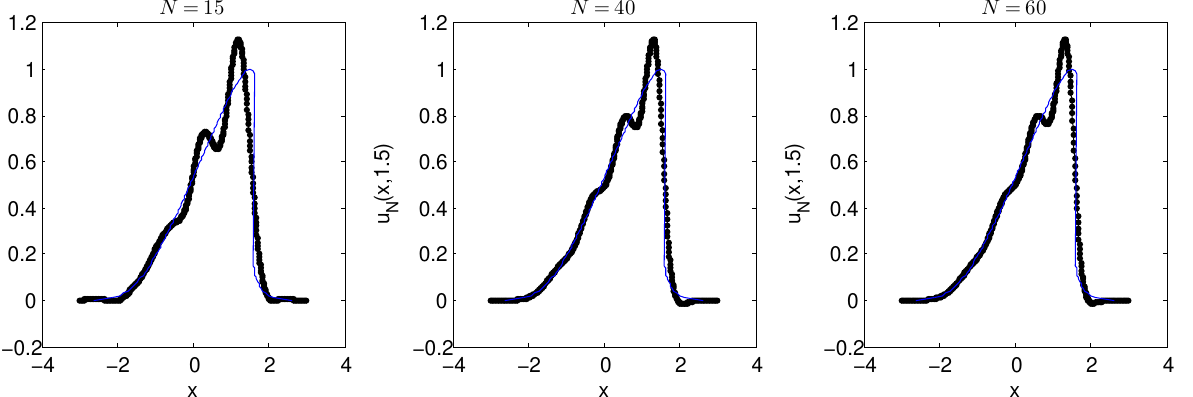}
	\caption{Solid blue line: the exact solution of inviscid Burger's equation; {dotted line: the {approximate solution of} the spectral scheme \eqref{eqn-spectral scheme} with $\a=2$, $\epsilon_N=0.5N^{-0.33}$, $m_N=\lfloor{5N^{0.16}}\rfloor$, ${\hat{q}_k^1}$ and $N=15,40,60$.}}\label{fig-withviscosity_q1}
\end{figure}

\begin{figure}[!h]
	\includegraphics[trim = 10mm 110mm 0mm 110mm, clip,scale=0.8]{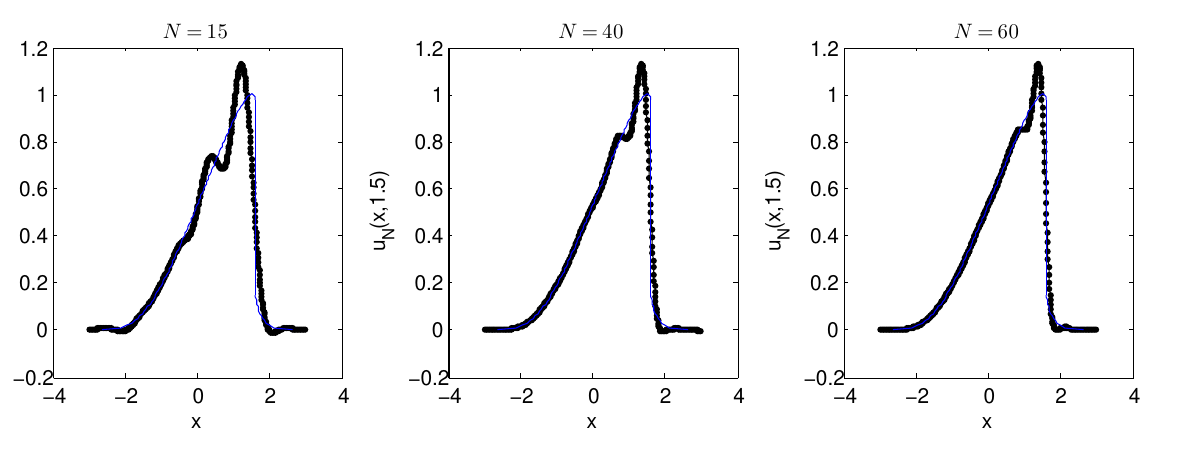}
	\caption{Solid blue line: the exact solution of inviscid Burger's equation; {dotted line: the {approximate solution of} the spectral scheme \eqref{eqn-spectral scheme} with $\a=2$, $\epsilon_N=0.5N^{-0.33}$, $m_N=\lfloor{5N^{0.16}}\rfloor$, ${\hat{q}_k}^2$ and $N=15,40,60$.}}\label{fig-withviscosity_q2}
\end{figure}

\begin{figure}[!h]
	\includegraphics[trim = 18mm 118mm 0mm 115mm, clip,scale=0.85]{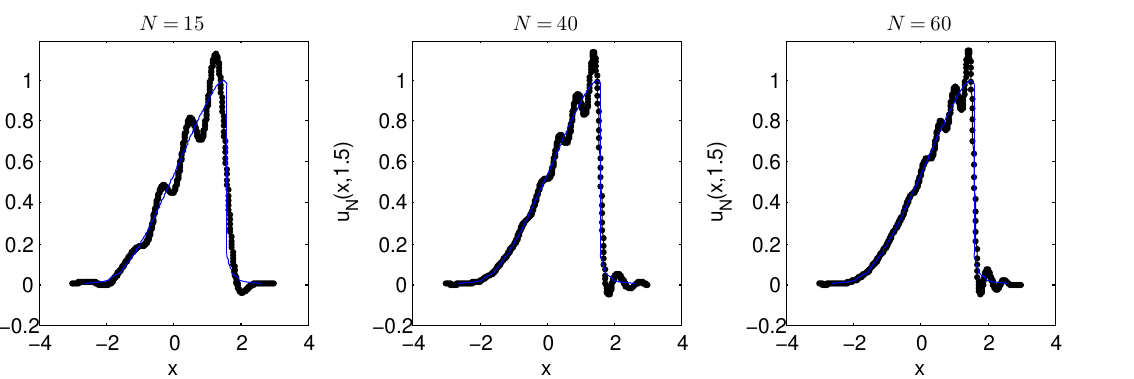}
	\caption{Solid blue line: the exact solution of inviscid Burger's equation; { dotted line: the {approximate solution of} the spectral scheme \eqref{eqn-spectral scheme} with $\a=2$, $\epsilon_N=0.5N^{-0.33}$, $m_N=\lfloor{5N^{0.16}}\rfloor$, ${\hat{q}_k}^3$ and $N=15,40,60$.}}\label{fig-withviscosity_q3}
\end{figure}

In {Figure \ref{fig-withviscosity_q1}-\ref{fig-withviscosity_q3}}, we add the viscosity term { $\hat{q}_k^1-\hat{q}_k^3$ in \eqref{eqn-qk}} by suitably tuning the parameters $\epsilon_N=0.5N^{-0.33}$ and $m_N=\lfloor{5N^{0.16}}\rfloor$ with $N=15,40$ and $60$, respectively{, where $\lfloor{\circ}\rfloor$ means the largest integer less than or equal to $\circ$. Compared with Figure 6.2 in \cite{AR}, Figure \ref{fig-withviscosity_q1}-\ref{fig-withviscosity_q3} show the similar situation with various $\hat{q}_k$. That it, the approximate solution with the least oscillations is given by the scheme \eqref{eqn-spectral scheme} with $\hat{q}_k^2$, while the best resolution of the shock is presented by that with $\hat{q}_k^3$.} The conditions on $m_N,\epsilon_N$ in Theorem \ref{thm-u} are satisfied. {Clearly, the convergence of the approximate solution} is better than {that} without viscosity.  {No matter what $\hat{q}_k$ is,} the oscillations do not alleviate as $N$ increases{, }but the discontinuity is resolved better with larger $N$. {Table \ref{table-1} list $||\D_xu_N||^2_{L^2(0,T;L^2(\R))}$, $||xu_N||^2_{L^2(0,T;L^2(\R))}$ and $||u_N||^2_{L^2(0,T;L^2(\R))}$ versus $N$, where the approximate solution $u_N$ is obtained by the spectral scheme \eqref{eqn-spectral scheme} with $\hat{q}_k^1$, $\a=\sqrt{2}$, $\e_N=0.5N^{-0.33}$.} It is used to numerically verify the condition \eqref{eqn-condition on xu_N} and the apriori estimate on $||\D_xQ_{m_N}u_N||_{L^2(0,T;L^2(\R))}^2$ \eqref{eqn-estimate of D_tQ}. The norm $||\circ||_{L^2(\R)}^2(t)$ at every time step is computed {on} the frequency side by Parseval's identity, and the integration in time in $||\circ||_{L^2(0,T;L^2(\R))}$ is performed by the trapezoid rule. The time steps are given by the adaptive algorithm {\em ode45} in Matlab. The command ``polyfit" in Matlab is used to find the minimal mean square linear fit of the growth rate of $||\D_xu_N||_{L^2(0,T;L^2(\R))}^2$,  $||xu_N||_{L^2(0,T;L^2(\R))}^2$ and  $||u_N||_{L^2(0,T;L^2(\R))}^2$ with respect to $N$, which are $N^{0.1420},N^{-0.0049}$ and $N^{0.0007}$, respectively. It numerically confirms that $||xu_N||_{L^2(0,T;L^2(\R))}^2\lesssim N^{-0.0049}$, which can be interpreted as {upper bound} independent of $N$, i.e. condition \eqref{eqn-condition on xu_N} is satisfied.  $||\D_xQ_{m_N}u_N||_{L^2(0,T;L^2(\R))}^2\lesssim N^{0.1420}+m_NN^{0.0007}\lesssim N^{0.1420}+N^{2*0.16+0.0007}\ll \frac1{\e_N}\approx N^{0.33}$, that is, the apriori estimate \eqref{eqn-estimate of D_tQ} is correct.

\begin{table}
\begin{center}
	\begin{tabular}{|c|c|c|c|c|c|c|c|}
\hline
		N&40&45&50&55&60&65&70\\
\hline
		$||\D_xu_N||_{L^2(0,T;L^2(\R))}^2$&$3.0621$&$3.1076$&$3.1504$&$3.1926$&$3.2350$&$3.2766$&$3.3145$\\
\hline
		$||xu_N||_{L^2(0,T;L^2(\R))}^2$&$0.9690$&$0.9682$&$0.9676$&$0.9671$&$0.9667$&$0.9665$&$0.9664$\\
\hline
	$||u_N||_{L^2(0,T;L^2(\R))}^2$&$1.8756$&$1.8757$&$1.8758$&$1.8759$&$1.8760$&$1.8762$&$1.8763$\\
\hline
	\end{tabular}
\caption{$||\D_xu_N||_{L^2(0,T;L^2(\R))}^2$, $||xu_N||_{L^2(0,T;L^2(\R))}^2$ and $||u_N||_{L^2(0,T;L^2(\R))}^2$ versus $N$, respectively, are displayed, where $u_N$ is the solution obtained by the spectral viscosity method \eqref{eqn-spectral scheme} with {$\hat{q}_k^1$,} $\e_N=0.5N^{-0.33}$.}\label{table-1}
\end{center}
\end{table}

\begin{figure}[!h]
    \includegraphics[trim =10mm 110mm 0mm 110mm, clip, scale=0.8]{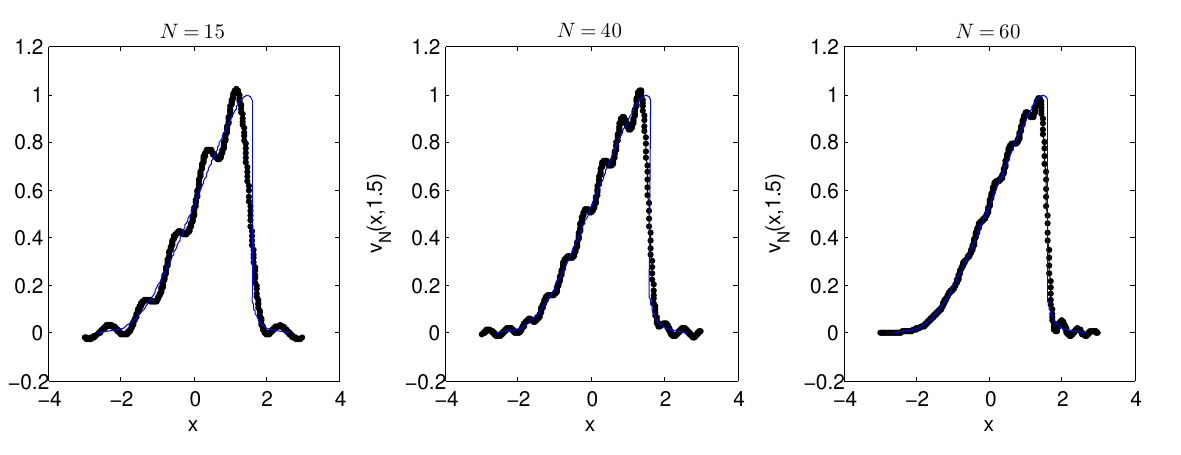}
	\caption{Solid blue line: the exact solution of inviscid Burger's equation; {dotted line: {the approximate solution of} the spectral scheme \eqref{eqn-spectral scheme-v} with $\epsilon_N=0.05N^{-0.33}$, $\a=\sqrt{2}$ and $N=15,40,60$.}}
\label{fig-withviscosity_differentN_alphasqrt2}
 \end{figure}

\begin{table}
\begin{center}
	\begin{tabular}{|c|c|c|c|c|c|c|c|}
\hline
		N&40&45&50&55&60&65&70\\
\hline
		$||\D_xv_N||_{L^2(0,T;L^2(\R))}^2$&$4.4442$&$4.5220$&$4.5099$&$4.4977$&$4.6272$&$4.9252$&$5.1875$\\
\hline
	$||v_N||^2_{L^2(0,T;L^2(\R))}$&$1.8829$&$1.8824$&$1.8814$&$1.8805$&$1.8804$&$1.8812$&$1.8819$\\
\hline
	$||x^2v_N||_{L^1(\R\times(0,T))}$&$1.9499$&$1.9153$&$1.8911$&$1.8671$&$1.8657$&$1.8844$&$1.8791$\\
\hline
	\end{tabular}
\caption{$||\D_xv_N||_{L^2(0,T;L^2(\R))}^2$, $||v_N||_{L^2(0,T;L^2(\R))}^2$ and $||x^2v_N||_{L^1(\R\times(0,T))}$ versus $N$, respectively, are displayed, where $u_N$ is the solution obtained by the spectral viscosity method \eqref{eqn-spectral scheme-v} with $\e_N=0.05N^{-0.33}$.}\label{table-2}
\end{center}
\end{table}

{Figure \ref{fig-withviscosity_differentalpha} reveals that the best resolution is obtained when $\a=\sqrt{2}$. Thus, we choose $\a=\sqrt{2}$ in the experiment of the scheme \eqref{eqn-spectral scheme-v}.} Figure  \ref{fig-withviscosity_differentN_alphasqrt2} displays the results of the spectral {scheme} \eqref{eqn-spectral scheme-v} with $\e_N=0.05N^{-0.33}$, {$\a=\sqrt{2}$} and $N=15,40,60$. It {shows} that the larger $N$ is, the better resolution at discontinuity is obtained. In Table \ref{table-2} we display $||\D_xv_N||_{L^2(0,T;L^2(\R))}^2$, $||v_N||_{L^2(0,T;L^2(\R))}^2$ and $||x^2v_N||_{L^1(\R\times(0,T))}$ versus $N$, respectively, where $v_N$ is the numerical solution to \eqref{eqn-spectral scheme-v} with $\e_N=0.05N^{-0.33}$. The norm $||\circ||_{L^2(0,T;L^2(\R))}^2$ is computed using the same rule as before. The {integration} in $||x^2v_N||_{L^1(\R\times(0,T))}$ is computed by trapezoid rule using equidistant grid. Again, we obtain the growth rate of  $||\D_xv_N||_{L^2(0,T;L^2(\R))}^2$,  $||v_N||_{L^2(0,T;L^2(\R))}^2$ and  {$||x^2v_N||_{L^1(\R\times(0,T))}$} versus $N$ by using ``polyfit" in Matlab, which are $N^{0.2431},N^{-0.0014}$ and $N^{-0.0639}$, respectively. The condition \eqref{cond-x2v_N} has been numerically verified.

{ \subsection{Different $\epsilon_N$}}

It is natural to ask how to tune $\e_N$ {in our scheme \eqref{eqn-spectral scheme}. We need to balance the resolution near the discontinuity and the oscillations away from the discontinuity in the choice of $\e_N$. With $N=40$, $\a=\sqrt{2}$ and $\hat{q}_k^1$, the results with $\e_N=0.5N^{-0.9},0.5N^{-0.33}$ and $0.5N^{-0.001}$ are plotted in Figure \ref{fig-withviscosity_differentepsilon_q1}. }  {It is not surprising that \eqref{eqn-spectral scheme} with $\e_N=0.5N^{-0.9}$ does not converge, due to the violation of condition $N^{-\frac12}\ll \e_N\ll N^{-2\theta}$, $0<\theta<\frac14$ in Theorem \ref{thm-u}, cf. Figure \ref{fig-withviscosity_differentepsilon_q1}. We expect that  the larger $\e_N$ will give a smoother approximate solution. However, as shown in Figure \ref{fig-withviscosity_differentepsilon_q1}, it is not the case, that is, the approximate solution with $\e_N=0.5N^{-0.33}$ yields fewer oscillations than that with $\e_N=0.5N^{-0.001}$. Similar situation can also be observed in Figure \ref{fig-withviscosity_differentepsilon_q2}-\ref{fig-withviscosity_differentepsilon_q3}, where the scheme \eqref{eqn-spectral scheme} with $\hat{q}_k^2$ and $\hat{q}_k^3$ instead.} 

\begin{figure}[!h]
    \includegraphics[trim = 5mm 115mm 0mm 110mm, clip, scale=.75]{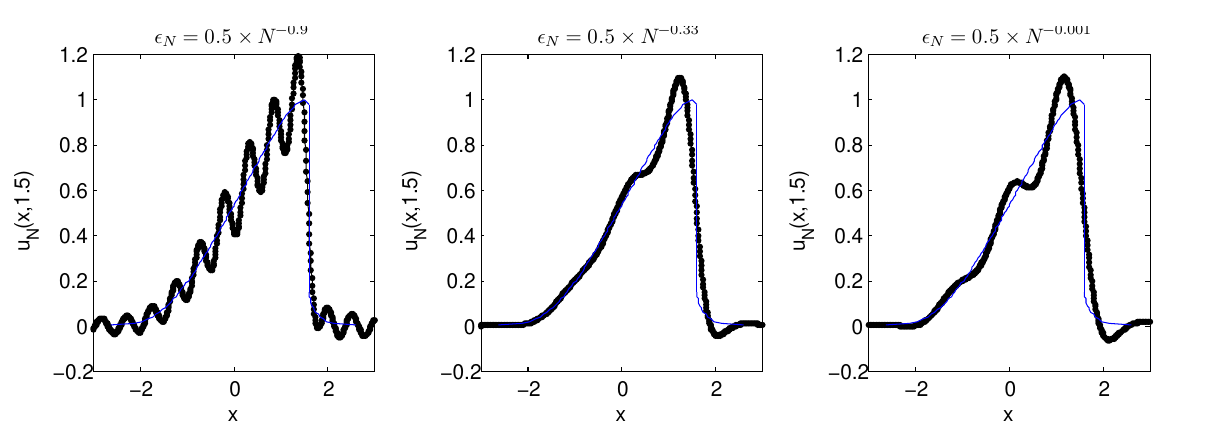}
	\caption{Solid blue line: the exact solution of inviscid Burger's equation; {dotted line: the spectral scheme \eqref{eqn-spectral scheme} with $N=40$, $\a=\sqrt{2}$, $\hat{q}_k^1$ and  $\epsilon_N=0.5N^{-0.9}, 0.5N^{-0.33}, 0.5N^{-0.001}$.}}
\label{fig-withviscosity_differentepsilon_q1}
 \end{figure}

\begin{figure}[!h]
    \includegraphics[trim = 5mm 115mm 0mm 110mm, clip, scale=.75]{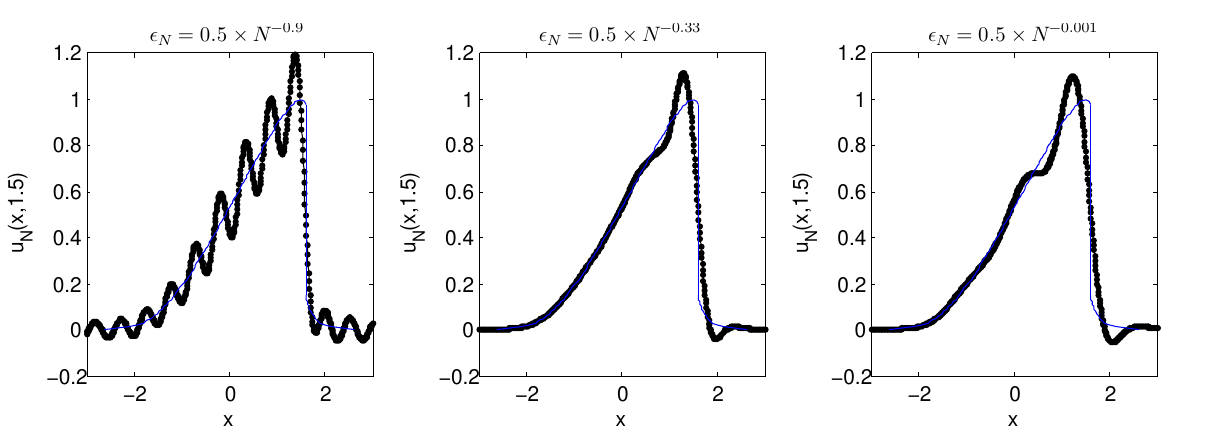}
	\caption{Solid blue line: the exact solution of inviscid Burger's equation; {dotted line: the spectral scheme \eqref{eqn-spectral scheme} with $N=40$, $\a=\sqrt{2}$, $\hat{q}_k^2$ and  $\epsilon_N=0.5N^{-0.9}, 0.5N^{-0.33}, 0.5N^{-0.001}$.}}
\label{fig-withviscosity_differentepsilon_q2}
 \end{figure}

\begin{figure}[!h]
    \includegraphics[trim = 5mm 115mm 0mm 110mm, clip, scale=.75]{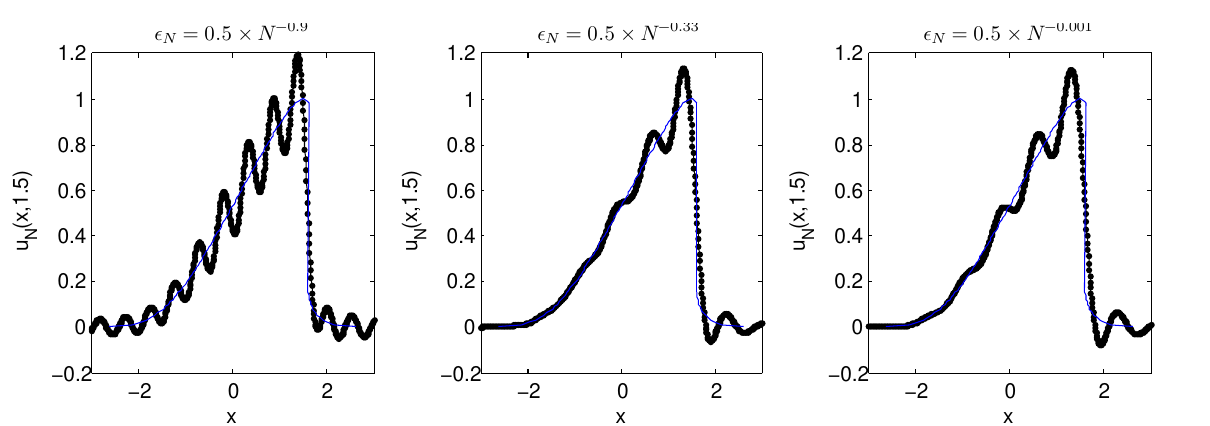}
	\caption{Solid blue line: the exact solution of inviscid Burger's equation; {dotted line: the spectral scheme \eqref{eqn-spectral scheme} with $N=40$, $\a=\sqrt{2}$, $\hat{q}_k^3$ and  $\epsilon_N=0.5N^{-0.9}, 0.5N^{-0.33}, 0.5N^{-0.001}$.}}
\label{fig-withviscosity_differentepsilon_q3}
 \end{figure}
\section{Conclusion}

In this paper, we propose two spectral viscosity methods based on the generalized Hermite functions for the solution of nonlinear scalar conservation laws in the whole line. Our {schemes have} been shown rigorously {that the approximate solutions} converge to the unique entropy solution by using compensated compactness arguments. The numerical experiments {of the inviscid Burger's equation} illustrate the implementability of our schemes. Thanks to the generalized Hermite functions, {the approximate solutions of our scheme have} fewer oscillations and better resolutions {of the discontinuity, with much smaller truncation modes $N$,} compared with those in \cite{AR}, even before adding the viscosity term.

\section*{Acknowledgements}
This project is sponsored by National Natural Science Foundation of China ({11501023}, 11471184) and Beijing Natural Science Foundation (1154011).

\end{document}